\documentclass{article}

\usepackage{arxiv}

\usepackage[utf8]{inputenc} 
\usepackage[T1]{fontenc}    
\usepackage{hyperref}       
\usepackage{url}            
\usepackage{booktabs}       
\usepackage{amsfonts}       
\usepackage{nicefrac}       
\usepackage{microtype}      
\usepackage{lipsum}
\usepackage{amssymb,amsmath,amsthm}
\usepackage{mathtools}
\usepackage{amsmath,
	amssymb,
	enumerate, 
	cmll}
\usepackage{color}
\usepackage{latexsym}
\usepackage{amsfonts}
\usepackage{graphicx}
\usepackage{hyperref, 
}
\usepackage{proof}
\usepackage{tikz}
\usepackage{xspace}

\newcommand{\NS}{\mathcal{S}}
\newtheorem{definition}{Definition}
\newtheorem{theorem}{Theorem}
\newtheorem{proposition}{Proposition}
\newtheorem{lemma}{Lemma}
\newtheorem{corollary}{Corollary}
\newtheorem{example}{Example}

\newcommand{\A}{\mathbf{A}}
\newcommand{\V}{\nu}
\newcommand{\conj}{\&}
\newcommand{\meta}{\Longrightarrow}
\newcommand{\Al}[1][A]{\ensuremath{\mathbf{#1}} }
\newcommand{\NA}{\textbf{A}}

\title{Nelson's Logic $\NS$}

\author{
  Thiago Nascimento \\
  Programa de P\'os-Gradua\c{c}\~ao em Sistemas e Computa\c{c}\~ao\\
  Universidade Federal do Rio Grande do Norte\\
  Natal, Brazil \\
  \texttt{thiagnascsilva@gmail.com} \\
   \And
 Umberto Rivieccio \\
  Departamento de Inform\'atica e Matem\'atica\\
  Universidade Federal do Rio Grande do Norte\\
  Natal, Brazil \\
   \And
  João Marcos \\
  Departamento de Inform\'atica e Matem\'atica\\
  Universidade Federal do Rio Grande do Norte\\
  Natal, Brazil \\
   \And
  Matthew Spinks \\
  Dipartimento di Pedagogia, Psicologia, Filosofia\\
  Universit\`{a}  di Cagliari\\
  Cagliari, Italy \\
}

\begin{document}
\maketitle
\begin{abstract}
Besides the better-known Nelson logic ($\mathcal{N}3$) 
and paraconsistent Nelson logic ($\mathcal{N}4$), in 1959
David Nelson introduced, with motivations of realizability and constructibility,
a logic 
called $\NS$.
The logic $\NS$ was originally presented by means of a calculus (crucially lacking the contraction rule) with infinitely many rule schemata and no semantics (other than the intended interpretation into Arithmetic). 
We look here  at the propositional fragment of~$\NS$, showing that it is algebraizable (in fact, implicative), in the sense of Blok and Pigozzi, 
with respect to a variety of three-potent involutive residuated lattices. We thus introduce the first known algebraic semantics for $\NS$ as well as a finite Hilbert-style calculus equivalent to Nelson's presentation; this also allows us to 
clarify
the relation between $\NS$ and the other two Nelson logics 
$\mathcal{N}3$
and $\mathcal{N}4$. 
\end{abstract}


\section{Introduction}

In the course of his 
extensive
investigations into the notion of `constructible falsity', David Nelson introduced
a number of systems of non-classical logics that have aroused considerable interest in the logic and algebraic
logic community (see, e.g., \cite{Od08} and the references cited therein).
Over the years, the main goal of Nelson's enterprise was to provide  logical formalisms 
that allow for more fine-grained analyses of notions such as `falsity' and `negation'
than either classical or intuitionistic logic can afford. 

Nelson's analysis of the meaning of `falsity' is in many ways analogous
to the intuitionistic analysis of `truth'.
The main property advocated by Nelson --- namely, 
if a formula $\neg (\phi \land \psi)$ is provable, then either $\neg \phi$ or $\neg \psi $ is provable --- is one that may be regarded as a dual to the well-known disjunction property of intuitionistic logic. 
In later investigations, just as the intuitionists argued against the usual object language formulation of the principle of excluded middle, $\phi \lor \neg \phi$, so Nelson was led to introduce logical systems that reject certain object language 
formulations of the principle of explosion (\emph{ex contradictione quodlibet}). 
The resulting logics thus combine an intuitionistic approach to truth with a dual-intuitionistic treatment of falsity, not unlike the one
of  the so-called bi-intuitionistic logic~\cite{Rau74b,Rau74}. 

The systems in the family nowadays known as Nelson's logics
share many properties with the positive fragment of intuitionistic logic 
(in particular, they do not validate Peirce's law $((\phi \Rightarrow \psi) \Rightarrow \phi) \Rightarrow \phi $).
They also possess 
a negation connective 
with 
inconsistency-tolerant features, in the sense that formulas such as $(\phi \land \neg \phi) \Rightarrow \psi$ need not be valid. 

The oldest and most well-known of Nelson's systems was introduced in \cite{Nel49} and is today known simply as \emph{Nelson logic} (following \cite{Od08}, we shall denote it by $\mathcal{N}3$). This logic is by now well
understood 
from a proof-theoretic (see, e.g., \cite{Me09a}) as well as
an algebraic point of view \cite{SpVe08, SpVe08a},
both perspectives allowing us to regard $\mathcal{N}3$ as a 
substructural logic in the sense of~\cite{GaJiKoOn07}.

\emph{Paraconsistent Nelson logic} $\mathcal{N}4$ is a weakening of $\mathcal{N}3$ 
introduced in \cite{AlNe84} 
(also independently considered  in \cite{LoEs72} and \cite{Rou74}) 
as, precisely, a non-explosive
version of $\mathcal{N}3$ suited for dealing with inexact predicates. 
Our understanding of the proof-theoretic as well as the algebraic properties of $\mathcal{N}4$ is more recent and 
still not thorough. 
However, thanks to recent results of M.\ Spinks and R.\ Veroff, 
$\mathcal{N}4$
can now be viewed as a member of the family of relevance logics; indeed, $\mathcal{N}4$
can be presented, to within definitional equivalence, as an axiomatic 
strengthening
of the contraction-free relevant logic $\mathcal{RW}$
(see~\cite{Spin18} for a summary of this work).

Thanks mainly to the works of S.\ Odintsov (see e.g.~\cite{Od08}, although the result has been formally stated for the first time only in~\cite{Ri11c}), we also know that $\mathcal{N}4$ (like  $\mathcal{N}3$) is algebraizable in the sense of Blok and Pigozzi.
This means that the consequence relation of $\mathcal{N}4$ can be completely characterized in terms of the equational consequence
of the corresponding algebraic semantics, which consists in a variety of algebras called  $\mathcal{N}4$-lattices. 

For our purposes, 
these algebraic completeness results entail in particular that we 
can compare both $\mathcal{N}$3 and $\mathcal{N}$4 to the logic $\NS$ --- the main object of the present paper --- by looking at the corresponding classes of algebras.
Before we turn our attention to $\NS$, 
let us dwell on another remarkable feature shared by  $\mathcal{N}3$ and $\mathcal{N}4$. 

For the propositional part (on which we shall exclusively focus in this paper), the 
language
of  both $\mathcal{N}3$ and $\mathcal{N}4$ comprises a conjunction $(\land)$, a disjunction $(\lor)$, a so-called `strong' negation 
$(\neg)$
and two implications: a so-called
`weak' $(\to)$ and a `strong' one $(\Rightarrow)$,
usually introduced 
via the following term:
$\phi \Rightarrow \psi := (\phi \to \psi) \land (\neg \psi \to \neg \phi)$. 
The presence of \emph{two} implications is crucial in Nelson's logics: it is this feature that 
makes, one may argue, the Nelson formalism more fine-grained than classical  or intuitionistic logic (or most many-valued logics, for that matter). 
With the two Nelson implications at hand, one is able to register finer shades of logical discrimination 
than it is possible in logics that are more `classically' oriented in nature (see Humberstone \cite{Humb05} for a general discussion of this issue).

In fact, different classical or intuitionistic tautologies may be proved within Nelson's logics using either $\to$ or $\Rightarrow$, creating a non-trivial interplay
between these two implications and 
with the negation connective;
the strong implication exhibits an 
inconsistency-tolerant behaviour, in that $( p \land \neg p) \Rightarrow q$ is not provable, 
while the other retains a more `classical' flavour, in that $( p \land \neg p) \rightarrow q$ turns out to be provable. 
On the other hand, while the weak implication ($\to$) allows us to see $\mathcal{N}3$ and $\mathcal{N}4$ as conservative expansions 
of positive intuitionistic logic 
by 
a negation connective with 
certain classical features (De Morgan, involutive laws),  the strong implication 
($\Rightarrow$) permits us to view their algebraic
counterparts as residuated structures, and therefore to regard $\mathcal{N}3$ and $\mathcal{N}4$ as strengthenings (as a matter of fact, axiomatic ones
)
of well-known substructural or relevance logics.

We note, in passing, that the overall picture is made more interesting and complex by the fact that  other  meaningful non-primitive connectives
can be defined --- for example, an `intuitionistic' negation (distinct from 
the primitive `strong' negation $\neg$) given by 
$\phi \to 0 $, 
or a `multiplicative' monoidal conjunction
given by $\neg (\phi \Rightarrow \neg \psi)$ --- and by the fact that interdefinability results hold even among the primitive connectives (some of these  being highly non-trivial to prove). We shall not enter into further details concerning this issue for this is not the main focus of the present paper; instead, we will now  turn our attention to the logic $\NS$, which has so far remained least well-known among the members of the Nelson family. 

The logic that we   (following Nelson's original terminology) call $\NS$ was introduced in~\cite{Nel59} with essentially the same motivations as $\mathcal{N}3$: that is, as a more flexible tool for the analysis of falsity, and in particular as an alternative to both $\mathcal{N}3$ and intuitionistic logic for interpreting Arithmetic through realizability. 
The propositional language of~$\NS$ comprises a conjunction, a disjunction, a falsity constant and (just one) implication. Whether this implication ought to be regarded as a `strong' or a `weak' one 
%
will become 
apparent 
as a result of the investigations in the present paper.

Nelson's presentation of~$\NS$ is given by means of a calculus that appears peculiar, to the modern eye, in several respects. It may look like a sequent calculus, but it is not. One could say that it is in fact a Hilbert-style calculus, though one with few axioms and many rules --- infinitely many, in fact: not just  instances, but infinitely many rule schemata. 
No standard semantics 
is provided in \cite{Nel59}  
for the calculus other than the intended interpretation of its (first-order) formulas  as arithmetic predicates. 

The above features may in part explain why $\NS$ has received, to the present day, very little attention in comparison to the other two Nelson's logics: to the point that, to the best of our knowledge, 
even the most basic questions about $\NS$ had not yet been asked,
let alone answered.  One could start by asking, for example, whether $\NS$ does admit a finite axiomatization. 
Another basic issue, which is interestingly obscured by Nelson's presentations  of~$\NS$ and $\mathcal{N}3$ in
\cite{Nel59}, is whether one of these two logics is stronger than the other or else whether they are incomparable.
Last but not least, 
Nelson observes that certain formulas are not provable in his system~\cite[p.~213]{Nel59}.
In the absence of a complete semantics for  $\NS$, it does not seem obvious how one could prove such claims.
Having established our algebraic completeness result, however, this will become quite straightforward.

The main motivation for the present paper has been to look at the above questions and, more generally --- taking advantage of the modern tools of algebraic logic --- to gain a better insight into
(the propositional part of)  $\NS$ and  into its
relation to other well-known non-classical logics. 
As we shall see in the following sections, we have 
successfully settled all the above-mentioned issues, and
the corresponding answers can be summarized as follows. 

First of all, 
$\NS$ may indeed be axiomatized by means of a finite Hilbert-style
calculus (having \textit{modus ponens} as its only rule schema) which is a strengthening of the contraction-free fragment of intuitionistic logic and also 
of the substructural logic known as the
Full Lambek Calculus with Exchange and Weakening ($\mathcal{FL}_{ew}$). This follows from our main result that $\NS$ is Blok-Pigozzi algebraizable (and therefore, enjoys a strong completeness
theorem) with respect to a certain class of residuated lattices, which are the canonical algebras associated with substructural logics
stronger than $\mathcal{FL}_{ew}$.
Furthermore, we may now say  that the implication of~$\NS$ is indeed a `strong' Nelson implication 
in the sense that 
it can be meaningfully compared with
the corresponding strong implications of $\mathcal{N}3$ and $\mathcal{N}4$.
From this vantage point, we will see that Nelson's logic $\mathcal{N}3$ may be regarded as an axiomatic strengthening  of~$\NS$, whereas $\mathcal{N}4$ is incomparable with~$\NS$. 
And finally we will confirm that
Nelson was correct in claiming that the formulas listed in~\cite[p.~213]{Nel59} are actually not 
provable
in~$\NS$. 

The paper is organized as follows. In Section~\ref{s:two} we introduce the logic $\NS$ 
through Nelson's original presentation (duly amending a number of obvious typos) and employ it to prove a few formulas that will be useful in the following sections.
In  Section~\ref{s:three} we prove that Nelson's calculus is algebraizable, and provide an axiomatization of the 
corresponding class of algebras, 
which 
we call \emph{$\NS$-algebras} (Subsection \ref{ss:threepointone}).  
Because of the above-mentioned peculiar features of  Nelson's calculus, the presentation of~$\NS$-algebras obtained algorithmically via the algebraization process 
is not very convenient.
We introduce then an alternative equational presentation in 
Subsection \ref{ss:threepointtwo} and show the equivalence of the two.
As a result of our own presentation, we 
establish that  $\NS$ is a strengthening 
of $\mathcal{FL}_{ew}$. Taking advantage of this insight, in Section~\ref{s:four} we introduce a finite
Hilbert-style calculus for $\NS$ which is simply an axiomatic strengthening  of a well-known calculus for $\mathcal{FL}_{ew}$. 
Completeness of our axiomatization, and therefore equivalence with Nelson's calculus, is obtained as a corollary
of the algebraizability results. In Section~\ref{s:five} we look at concrete $\NS$-algebras which provide counter-examples for  
the formulas  Nelson claimed to be unprovable within~$\NS$.
We present in Subsection~\ref{ss:dou} an easy way of building an $\NS$-algebra starting from a residuated lattice, which turns out to be useful later on (Section~\ref{s:ext}). 
Section~\ref{s:six} establishes the relation between~$\NS$ and the two other 
Nelson's logics, 
$\mathcal{N}3$ (Subsection \ref{ss:sixpointone}) and $\mathcal{N}4$ (Subsection \ref{ss:sixpointtwo}).
We show in particular that both $\mathcal{N}3$ and 
the three-valued \L ukasiewicz logic (but no other logic in the \L ukasiewicz family) may be seen as axiomatic strengthenings of~$\NS$. In Section~\ref{s:ext} we use the algebraic 
insights 
gained so far to obtain  information on the cardinality
of the strengthenings of~$\NS$.
Finally, Section~\ref{s:seven} contains suggestions for future work.

The present paper is an expanded and improved version of~\cite{NaRi}, 
to which we shall refer whenever doing so allows us to omit or shorten our proofs.
Let us highlight the main differences and present novelties. 
From Section~\ref{s:two} to Theorem~\ref{th:ddt}
of Section~\ref{s:four} we follow essentially Sections 2--4 of~\cite{NaRi}.
The remaining part of Section~\ref{s:four} (dealing with EDPC and WBSO varieties) is new, as is
Section~\ref{s:five}. In particular,
Subsection~\ref{ss:fiveone} contains a proof of the claim made in~\cite{NaRi} that the distributivity axiom (as well as the other
formulas mentioned in Proposition~\ref{p:51}) is not valid in~$\NS$. The usage of the Galatos-Raftery doubling construction
(in both Section~\ref{s:five} and  Section~\ref{s:six}) is entirely new.
Subsection~\ref{ss:sixpointone} is essentially an expanded version of Section 5.1 from~\cite{NaRi};  
on the other hand, the results from Proposition~\ref{prop:n3ext} to the end of Section~\ref{s:six} are new. 
Section~\ref{s:ext} is also entirely new.
\newpage
\section{Nelson's Logic $\NS$}
\label{s:two}

In this section we recall Nelson's 
original presentation of the propositional fragment of~$\NS$, 
modulo the correction of a number of typos that appear 
in~\cite{Nel59}. 


We  denote by $\mathbf{Fm}$ the formula algebra over a given similarity type, freely generated by a denumerable set of propositional variables $\{p,q,r,\ldots\}$. 
We denote by $Fm$ the carrier of $\mathbf{Fm}$, and use $\varphi$, $\psi$ 
and~$\gamma$, possibly decorated with subscripts, to refer to arbitrary elements of $Fm$.
A \emph{logic} is then defined as a substitution-invariant conse\-quence relation~$\vdash$ on $Fm$.  


\begin{definition}
	\label{nelsons}
	Nelson's logic $\NS := \langle \mathbf{Fm}, \vdash_{\NS} \rangle$ is the sentential logic in the language $\langle \land, \lor, \Rightarrow, \neg, 0 
	\rangle $ of type $\langle 2, 2, 2, 1, 0 \rangle$
	defined by the Hilbert-style calculus with the rule
	schemata in Table~\ref{srules} and the following axiom schemata.  We shall henceforth use the abbreviations
	$\phi \Leftrightarrow \psi: = (\phi \Rightarrow \psi) \land (\psi \Rightarrow \phi )$
	and
	$1 := \neg 0$.

	\medskip
	\medskip
	
	
	\noindent 
	Axioms 
	
	\begin{description}
		\item[\texttt{(A1)}]\quad $\phi \Rightarrow \phi$ 
		\item[\texttt{(A2)}]\quad  $0 
		\Rightarrow \psi$
		\item[\texttt{(A3)}]\quad  $\neg \phi \Rightarrow (\phi \Rightarrow 0 
		)$
		\item[\texttt{(A4)}]\quad  $
		1 
		$
		\item[\texttt{(A5)}]\quad  $(\phi \Rightarrow \psi) \Leftrightarrow (\neg \psi \Rightarrow \neg \phi)$
		
	\end{description}
\end{definition}

In Table \ref{srules} below, following Nelson's notation,  $\Gamma$ denotes an arbitrary finite list $( \phi_1, \ldots, \phi_n )$ of formulas, and the following abbreviations are used:
$$
\Gamma \Rightarrow \phi \ := \ \phi_{1} \Rightarrow (\phi_{2} \Rightarrow (\ldots  \Rightarrow (\phi_{n} \Rightarrow \phi)\ldots)).
$$ 
If $\Gamma$ is empty, then  $\Gamma \Rightarrow \phi$ is just $\phi$. 
Moreover, we let
$$
\phi \Rightarrow^{2} \psi \ := \ \phi \Rightarrow (\phi \Rightarrow \psi)
$$ 
and 
$$
\Gamma \Rightarrow^{2} \phi \ := \ \phi_{1} \Rightarrow^{2} (\phi_{2} \Rightarrow^{2}(\ldots  \Rightarrow^{2} (\phi_{n} \Rightarrow^{2} \phi)\ldots)).
$$

\begin{table}[h]
	\begin{tabular}{lll}	
		\infer[\texttt{(P)}]{\Gamma \Rightarrow (\psi \Rightarrow (\phi \Rightarrow \gamma))}{\Gamma \Rightarrow (\phi \Rightarrow (\psi \Rightarrow \gamma))} &  
		\infer[\texttt{(C)}]{\phi \Rightarrow (\phi \Rightarrow \gamma)}{\phi \Rightarrow (\phi \Rightarrow (\phi \Rightarrow \gamma))} &
		\infer[\texttt{(E)}]{\Gamma \Rightarrow \gamma }{\Gamma \Rightarrow \phi  & \phi \Rightarrow \gamma} \\ \\
		
		\infer[\mathrm{(\Rightarrow \texttt{l})}]
		{\Gamma \Rightarrow ((\phi \Rightarrow \psi) \Rightarrow \gamma)}{\Gamma \Rightarrow \phi  & & & \psi \Rightarrow \gamma} &   \infer[\mathrm{(\Rightarrow \texttt{r})}]{\phi \Rightarrow \gamma}{\gamma}  &
		
		\infer[\mathrm{(\land \texttt{l1})}]{(\phi \land \psi) \Rightarrow \gamma}{\phi \Rightarrow \gamma} \\ \\ \infer[\mathrm{(\land \texttt{l2})}]{(\phi \land \psi) \Rightarrow \gamma}{\psi \Rightarrow \gamma}  & \infer[\mathrm{(\land \texttt{r})}]
		{\Gamma \Rightarrow (\phi \land \psi)}{\Gamma \Rightarrow \phi & & \Gamma \Rightarrow \psi}   &
		\infer[\mathrm{(\lor \texttt{l1})}] 
		{(\phi \lor \psi) \Rightarrow \gamma }{\phi \Rightarrow \gamma  & & \psi \Rightarrow \gamma} \\ \\ \infer[\mathrm{(\lor \texttt{l2})}] 
		{(\phi \lor \psi) \Rightarrow^2 \gamma}{\phi \Rightarrow^2 \gamma & & \psi \Rightarrow^2 \gamma} & \infer[\mathrm{(\lor \texttt{r1} )}]{\Gamma \Rightarrow (\phi \lor \psi)}{\Gamma \Rightarrow \phi}  & \infer[\mathrm{(\lor \texttt{r2})}]{\Gamma \Rightarrow (\phi \lor \psi)}{\Gamma \Rightarrow \psi} \\ \\
		
		\infer[\mathrm{(\neg {\Rightarrow} \texttt{l})}]{\neg(\phi \Rightarrow \psi) \Rightarrow \gamma}{(\phi \land \neg \psi) \Rightarrow \gamma} & \infer[\mathrm{(\neg {\Rightarrow} \texttt{r})}]{\Gamma \Rightarrow^{2} \neg (\phi \Rightarrow \psi)}{\Gamma \Rightarrow^{2} (\phi \land \neg \psi)}  &
		
		\infer[\mathrm{(\neg {\land} \texttt{l})}]{\neg(\phi \land \psi) \Rightarrow \gamma }{(\neg \phi \lor \neg \psi) \Rightarrow \gamma} \\ \\ \infer[\mathrm{(\neg {\land}\texttt{r})}]{\Gamma \Rightarrow \neg (\phi \land \psi)}{\Gamma \Rightarrow (\neg \phi \lor \neg \psi)}  &
		\infer[\mathrm{(\neg {\lor} \texttt{l})}]{\neg (\phi \lor \psi) \Rightarrow \gamma}{(\neg \phi \land \neg \psi) \Rightarrow \gamma} &  \infer[\mathrm{(\neg{\lor}\texttt{r})}]{\Gamma \Rightarrow \neg(\phi \lor \psi)}{\Gamma \Rightarrow (\neg \phi \land \neg \psi)}  \\ \\
		\infer[\mathrm{(\neg \neg \texttt{l})}]{\neg \neg \phi \Rightarrow \gamma}{\phi \Rightarrow \gamma} & \infer[\mathrm{(\neg \neg\texttt{r})}]{\Gamma \Rightarrow \neg \neg \phi}{\Gamma \Rightarrow \phi}
		\\
	\end{tabular}
	\caption{Rules of~$\NS$}
	\label{srules}
\end{table}	

We have fixed obvious typos in the rules
$\mathrm{(\land \texttt{l2})}$, $\mathrm{(\land \texttt{r})}$ and  $\mathrm{(\neg {\Rightarrow} \texttt{r})}$ as they appear in \cite[p.~214-5]{Nel59}. For example, rule $\mathrm{(\neg {\Rightarrow} \texttt{r})}$  from Nelson's paper reads as: 
$$
\infer[]{\Gamma \Rightarrow^{2} (\phi \land  \psi)}{\Gamma \Rightarrow^{2} \neg (\phi \Rightarrow \psi)}
$$
This is not even classically valid. One might consider correcting the rule as follows:
$$
\infer[]{\Gamma \Rightarrow^{2} (\phi \land \neg \psi)}{\Gamma \Rightarrow^{2} \neg (\phi \Rightarrow \psi)}
$$
but this  does not seem consistent with the 
convention 
used by Nelson for the other rules: the $\Rightarrow$ connective should appear on the right-hand side at the bottom, and  $\land$  at the top. 
We assume thus that this corrected version was intended to have been written upside-down.
%
The rule \texttt{(C)}, called
\emph{weak condensation}
by Nelson,
replaces (and is indeed a weaker form of) the 
\emph{contraction} rule: 
$$\infer{\phi \Rightarrow \psi}{\phi \Rightarrow (\phi \Rightarrow \psi)}$$
%
This rule
is also known  in the literature as `$3$-$2$ contraction' 
\cite[p.~389]{Re93c}
and corresponds, 
on algebraic models, 
to the 
property of \emph{three-potency} (see Section~\ref{ss:threepointtwo}).
Notice also that the usual rule of \textit{modus ponens} 
(from $\phi$ and $\phi \Rightarrow \psi$, infer $\psi$)
is an instance of  
$\mathrm{\texttt{(E)}}$ for  $\Gamma = \varnothing$. %
%
Lastly,  let us highlight  that every rule schema involving $\Gamma$ is actually a shorthand for a denumerably infinite set of rule schemata. For instance, the schema:
$$ 
\infer[\mathrm{(\Rightarrow \texttt{l})}]
{\Gamma \Rightarrow ((\phi \Rightarrow \psi) \Rightarrow \gamma)}{\Gamma \Rightarrow \phi  & & & \psi \Rightarrow \gamma}
$$
stands for the following collection of rule schemata:
\begin{description}
	\item[] \quad $\infer
	{ (\phi \Rightarrow \psi) \Rightarrow \gamma}{  \phi  & & \psi \Rightarrow \gamma} $
	\smallskip
	
	\item[]  \quad 
	$\infer
	{\gamma_{1} \Rightarrow ((\phi \Rightarrow \psi) \Rightarrow \gamma)}{ \gamma_{1} \Rightarrow \phi  & & & \psi \Rightarrow \gamma}$
	\smallskip
	
	\item[] \quad $\cdots$ 
	\smallskip
	
	\item[] \quad $\infer
	{\gamma_{1} \Rightarrow (\gamma_{2} \Rightarrow (\gamma_{3} \Rightarrow ((\phi \Rightarrow \psi) \Rightarrow \gamma)))}{\gamma_{1} \Rightarrow( \gamma_{2}  \Rightarrow (\gamma_{3} \Rightarrow \phi))  & & & \psi \Rightarrow \gamma}$
	\smallskip
	
	\item[] \quad $\cdots$ 
	
\end{description}
Thus, Nelson's calculus employs not just infinitely many axiom and rule instances, but actually
infinitely many rule schemata. Notice, nonetheless, that defining as usual a derivation as a 
finite sequence of formulas, we have that the consequence relation of $\NS$ is finitary.

\section{Algebraic semantics}
\label{s:three}
In this section we show that $\NS$ is algebraizable 
(and, in fact, is \emph{implicative} in Rasiowa's sense \cite[Definition 2.3]{Font16}), and we give two equivalent presentations for its equivalent algebraic semantics (that we shall call \emph{$\NS$-algebras}). 
The first presentation
is obtained via the algorithm of \cite[Theorem 2.17]{BP89}, while the second one  is closer to the usual axiomatizations
of classes of residuated lattices, which constitute the algebraic counterparts of many logics in the substructural family. In fact, the latter presentation of~$\NS$-algebras will allow us to see at a glance that  
they form an equational class, and will also make it 
easier to compare them with other known classes of algebras related to substructural logics.

Following standard usage, we denote by  $\A$ (in boldface) an algebra and by $A$ (italics) its carrier set.
Given the formula algebra $\mathbf{Fm}$, the associated set of \emph{equations}, $Fm \times Fm$, will henceforth be denoted by $Eq$. 
To say that $\langle \phi, \psi \rangle \in Eq$, we will write $\phi \approx \psi$, as usual. We say that a valuation 
$\V \colon Fm \to
A $ \emph{satisfies $\phi \approx  \psi$ in $\A$} when $\V(\phi) = \V(\psi)$. We say that an algebra $\A$ \emph{satisfies $\phi \approx  \psi$} when all valuations 
over $\A$
satisfy it.

%
%
%

It will be convenient for us to work with the following definition of \emph{algebraizable logic}, which is not the original 
one~\cite[Definition 2.1]{BP89} but an equivalent so-called intrinsic characterization~\cite[Theorem 3.21]{BP89} of it:

\begin{definition}\label{def-algebraizable}	
	A logic $\mathcal{L}$ is \emph{algebraizable} if and only if there are equations  $\mathsf{E}(x) \subseteq Eq$ and formulas $\Delta(x, y) \subseteq Fm$ such that:
	
	\vspace{0.2cm}
	\begin{tabular}{ll}
		$\mathbf{(R)}$	& %
		$\varnothing \vdash_{\mathcal{L}} \Delta(\phi, \phi)$
		\\
		$\mathbf{(Sym)}$	& $\Delta(\phi, \psi) \vdash_{\mathcal{L}} \Delta(\psi, \phi)$\\
		$\mathbf{(Trans)}$	& $\Delta(\phi, \psi) \cup \Delta(\psi, \gamma) \vdash_{\mathcal{L}} \Delta(\phi, \gamma)$\\
		$\mathbf{(Rep)}$	& $\bigcup_{i=1}^n \Delta(\phi_{i}, \psi_{i}) \vdash_{\mathcal{L}} \Delta(\bullet( \phi_{1},\ldots,\phi_{n}), \bullet( \psi_{1},\ldots, \psi_{n}))$, \\
		& for each $n$-ary connective $\bullet$ \\
		$\mathbf{(Alg3)}$	& $\phi \dashv \vdash_{\mathcal{L}} \Delta(\mathsf{E}(\phi))$
	\end{tabular} 
	\vspace{0.1cm}
	
	%
	\noindent 
\end{definition}
	Here, the notation $\Gamma \vdash \Delta$, where $\Delta$ is a set of formulas, means that $\Gamma \vdash \phi$ for each $\phi \in \Delta$.
	The set $\mathsf{E}(x)$ is said to be the set of \emph{defining equations} and $\Delta(x, y)$ is said to be the set of \emph{equivalence formulas}.
	We say that $\mathcal{L}$ is \emph{implicative} when it is algebraizable with  
	$\mathsf{E}(x) := \{ x \approx \alpha(x, x) \}$ and
	$\Delta(x,y) := \{ \alpha(x, y), \alpha(y, x) \}$, where  
	$\alpha(x, y)$ denotes a binary term in the language of $\mathcal{L}$.
	In such a case, the term $\alpha(x, x)$ determines an algebraic constant on every algebra belonging to the algebraic counterpart
	of $\mathcal{L}$ ({see \cite[Lemma 2.6]{Font16}}), and is usually denoted accordingly.

\begin{theorem}[\cite{NaRi}, Theorem 1]
	\label{th:eqfor}
	The logic ${\NS}$ is implicative, and thus algebraizable,
	with defining equation $\mathsf{E}(x) := \{x \approx 1\}$ --- or, equivalently, $\mathsf{E}(x) := \{x \approx x \Rightarrow x \}$ --- and equivalence formulas $\Delta(x, y) := \{x \Rightarrow y, y \Rightarrow x\}.$
	
\end{theorem}

\subsection{$\NS$-algebras}
\label{ss:threepointone}

By Blok and Pigozzi's  algorithm (\cite[Theorem 2.17]{BP89};  see also 
\cite[Theorem 30]{CzPi04},
\cite[Proposition 3.44]{Font16}), 
the equivalent algebraic
semantics of~$\NS$ is the 
quasivariety of algebras~\cite[Definition V.2.24]{BuSa00} given by
the following definition:


\begin{definition}
	\label{def:salg}
	An \emph{$\NS$-algebra} is a structure $\mathbf{A} := \langle A, \land, \lor, \Rightarrow, \neg, 0, 1 \rangle$ of type $\langle 2, 2, 2, 1, 0, 0\rangle$ that satisfies 
	the following
	equations and quasiequations:

	\begin{enumerate}
		\item For each axiom $\varphi$ of 
		${\NS}$, the equation $\mathsf{E}(\varphi)$ defined as  
		$
		\mathsf{E}(\varphi)  : =  \varphi \approx 1
		$. 
		
		\item 
		$x \Rightarrow x \approx 1 
		$.
		
		
		\item For each rule
		$$
		\infer[\quad \mathbf{(R)}]{\phi}{\varphi_1 & \cdots & \varphi_n}
		$$
		of~$\NS$, the quasiequation $\mathsf{Q}(\texttt{R})$ defined as follows:
		$$
		\mathsf{Q}(\texttt{R}) \ : = \ \left[\varphi_1 \approx 1 \;\; \conj \; \ldots \ \conj \;\; \varphi_n \approx 1\right] 
		\meta \phi \approx 1.
		$$    
		
		\item
		$ 
		\left[x \Rightarrow y \approx 1\;\; \conj \;\; y \Rightarrow x \approx 1\right] \meta x \approx y$.
	\end{enumerate}
\end{definition}

%
%
%
%
%
We shall henceforth denote by $\mathsf{E} ( \texttt{An} )$ 
the equation %
given in Definition \ref{def:salg}.1
for the  
axiom $\texttt{An}$ (for $ 1 \leq \mathbf{n} \leq 5 $) of~$\NS$, 
and by $\mathsf{Q}(\texttt{R})$ the quasiequation given in Definition \ref{def:salg}.3
for the 
rule $\texttt{R}$ of~$\NS$.
We will  also
use the following abbreviations:
$a* b := \neg (a\Rightarrow \neg b)$, \
$a^2 := a* a$ and
\	$a^n : = a* (a^{n-1})$ for $n > 2$.
As the notation suggests, the defined connective $*$ may be regarded as a `multiplicative conjunction' in the sense of substructural logics. On $\NS$-algebras, the operation~$*$ will be interpreted as a monoid operation having the implication ($\Rightarrow$) as its residuum, whereas the `additive conjunction' $\land$ will be interpreted as the meet of the underlying lattice structure.
We list next a few properties of~$\NS$-algebras that will help us in viewing them, later on, as a class of residuated structures:

\begin{proposition}[\cite{NaRi}, Proposition 3] 
	Let  $\mathbf{A} := \langle A, \land, \lor, \Rightarrow, \neg, 0, 1 \rangle$ be an $\NS$-algebra and $a,b, c \in A$. Then:
	\label{prop:coisas}
	
	%
	%
	%
	%
	%
	%
	%
	%
	%
	%
	%
	%
	%
	%
	%
	%
	\begin{enumerate}
		\item $\langle A, \land, \lor , 0, 1 \rangle$ is a bounded lattice whose order $\leq$ is
		given by 
		$a \leq b $ iff $a \Rightarrow b = 1 $.
		\item  $\langle A,*,1\rangle$ is a commutative monoid.
		
		\item 	The pair $(*, \Rightarrow)$ is residuated with respect to $\leq$, i.e., 
		\
		$
		a * b \leq c \  \textrm{ iff } \ a \leq b \Rightarrow c
		$. 
		
		\item $  a \Rightarrow b = \neg b \Rightarrow \neg a$.

		\item$ a \Rightarrow 0 = \neg a $  and  $\neg \neg a = a $.

		\item $a^2 \leq a^3$.

		\item $(a \lor b)^2 \leq a^2 \lor b^2$.
		
	\end{enumerate}
	
\end{proposition}

\subsection{$\NS$-algebras as residuated lattices}
\label{ss:threepointtwo}

In this section we introduce an equivalent presentation of~$\NS$-algebras which takes precisely the properties in Proposition \ref{prop:coisas}  as postulates. 
We begin by recalling the following well-known definitions (see e.g.~\cite[p.~185]{GaJiKoOn07}): 

\begin{definition}
	A \emph{commutative integral 
		residuated lattice (CIRL)} is an algebra 
	$\mathbf{A} := \langle A, \land, \lor, *, \Rightarrow, 0, 1 \rangle$ of type $\langle 2 ,2 ,2, 2, 0, 0 \rangle$ such that:

	\begin{enumerate}
		\item $\langle A, \land, \lor \rangle$  is a  lattice (with ordering $ \leq$) with 
		maximum element 1.
		
		\item $\langle A, *, 1 \rangle$ is a commutative monoid.
		
		\item $(*, \Rightarrow)$ forms a residuated pair with respect to $\leq$, that is: $a*b \leq c$ \ iff \ $ b \leq a \Rightarrow c$ \  for all $a,b,c \in A$. 
	\end{enumerate}
\end{definition}
	\noindent
	We say that a CIRL
	is \emph{three-potent}\footnote{The reader should be advised, however, that for some authors, for example~\cite[p.~96]{GaJiKoOn07}, three-potency corresponds to the equation $x^3 \approx x^4$ and two-potency to $x^2 \approx x^3$.}
	when  $a^2 \leq a^3$ for all $a \in \mathbf{A}$ (in which case it follows that
	$a^2 = a^3$). 
	If the lattice ordering of $\mathbf{A}$ also has $0$ as a minimum element, then  $\mathbf{A}$ is a 
	\emph{commutative integral bounded 
		residuated lattice} (\emph{CIBRL}). Setting $\neg a:= a\Rightarrow 0$, we then  say that a 
	CIBRL
	is \emph{involutive}  when it satisfies the equation $\neg \neg x \approx x$ \cite[p.~186]{GaRa04}. The latter last equation implies that
	$ x \Rightarrow y \,\approx\, \neg y \Rightarrow {\neg} x$ \cite[Lemma 3.1]{On10}. 

The property of integrality mentioned in the above definition corresponds to the requirement that~$1$ be at the same time the neutral element of the monoid and the top element of the lattice order. One easily sees that integrality entails that the operation $*$ is $\leq$-decreasing ($a * b \leq a$) and that the term $x \Rightarrow x$ defines thus an algebraic constant in the lattice which is interpreted as $1$.	

\begin{definition}
	\label{def:slinha}
	An \emph{$\NS'$-algebra} is a three-potent involutive CIBRL. 
\end{definition}

Since CIBRLs 
form an equational class  \cite[Theorem 2.7]{GaJiKoOn07}, it is clear that
$\NS'$-algebras are also an equational class. By contrast, from Definition~\ref{def:salg} it is far from obvious  whether
$\NS$-algebras are equationally axiomatizable or not. 
By Proposition \ref{prop:coisas}, though, we immediately obtain the following result:

\begin{proposition}
	\label{prop:slinha}
	Let $\mathbf{A} := \langle A, \land, \lor, \Rightarrow, \neg, 0, 1 \rangle$ be an $\NS$-algebra. Setting $x*y := \neg (x \Rightarrow \neg y)$, 
	we have that $\mathbf{A}^\prime := \langle A, \land, \lor, *, \Rightarrow,  0, 1 \rangle$ is an $\NS'$-algebra.
\end{proposition}


The next lemma will allow us to verify that, conversely,
every $\NS'$-algebra has a term-definable  $\NS$-algebra structure. 
Thus, as anticipated,
$\NS'$-algebras and 	$\NS$-algebras 
can be viewed as two presentations (in slightly different languages) of the same class of abstract structures.
To establish this we shall check that
every $\NS'$-algebra satisfies all (quasi)equations introduced in Definition~\ref{def:salg}. 


\begin{lemma}[\cite{NaRi}, Lemma 1] 
	\label{lemma1} 
	\begin{enumerate}
		\item	Any CIRL 
		satisfies the equation $ (x\lor y)*z \approx (x*z) \lor (y*z)$.
		\item Any CIRL 
		satisfies 
		$x^{2} \lor y^{2} \approx (x^{2} \lor y^{2})^{2}$.
		\item  Any three-potent CIRL 
		satisfies 
		$(x\lor y^{2})^{2} \approx (x\lor y)^{2}.$
		\item Any three-potent CIRL 
		satisfies  $( x\lor y)^{2} \approx x^{2} \lor y^{2}$.
	\end{enumerate}
\end{lemma}

\begin{proposition}
	\label{p:slinhaiss} 
	Let $\mathbf{A}^\prime := \langle A, \land, \lor, *, \Rightarrow, 0, 1 \rangle$ be an $\NS'$-algebra. Setting $\neg x := x \Rightarrow 0$, 
	we have that  $\mathbf{A} := \langle A, \land, \lor, \Rightarrow,  \neg, 0, 1 \rangle$ is an $\NS$-algebra.		
\end{proposition} 

\begin{proof}
	Let $\mathbf{A}^\prime$ be an $\NS'$-algebra. 
	We first consider the equations %
	obtained from Definition~\ref{def:salg}.1.
	To check $\mathsf{E} ( \texttt{A1} )$ (namely, the equation $x \Rightarrow x \approx 1$) one may use residuation and the facts that $1*a \leq a$ and that~$1$ is the maximum element.  
	$\mathsf{E} ( \texttt{A2} )$ follows from the fact that $0$ is the minimum element of $\mathbf{A}^\prime$.
	$\mathsf{E} ( \texttt{A3} )$ follows from the definition of $\neg$ in $\NS'$ and from $\mathsf{E} ( \texttt{A1} )$.
	$\mathsf{E} ( \texttt{A4} )$ follows from the fact that %
	$1 \approx 1 \Rightarrow 1$. 
	$\mathsf{E} ( \texttt{A5} )$ follows from the fact that $\mathbf{A}^\prime$  is  involutive. 
	%
	We look next at the quasiequations 
	obtained from Definition \ref{def:salg}.3:
	\smallskip
	
	$\mathsf{Q}(\texttt{P})$ follows from  the commutativity of $*$ and from the equation 
	$(a *b ) \Rightarrow c \approx a \Rightarrow (b \Rightarrow c)$. \smallskip
	
	$\mathsf{Q}(\texttt{C})$ follows from 3-potency: since  $ a^2 \leq a^3$, we have that
	$a^3 \Rightarrow b \approx 1$ implies $a^2 \Rightarrow b \approx 1$. 
	
	$\mathsf{Q}(\texttt{E})$ follows from the fact that $\mathbf{A}^\prime$ carries a partial order $\leq$ that is determined by the implication $\Rightarrow$.
	\smallskip
	
	To prove $\mathsf{Q}(\mathrm{\Rightarrow \texttt{l}})$, suppose $a \leq b$ and $c \leq d$. From $c \leq d$, as $b \Rightarrow c \leq  b \Rightarrow c$, using residuation we have that $b*(b \Rightarrow c) \leq c$, thus $b*(b \Rightarrow c) \leq d$ and therefore $b \Rightarrow c \leq b \Rightarrow d$. As $a \leq b$, using residuation and the $\leq$-monotonicity of~$*$ we have that $a*(b \Rightarrow d) \leq b*(b \Rightarrow d) \leq d$, therefore $b \Rightarrow d \leq a \Rightarrow d$ and thus $b \Rightarrow c \leq a \Rightarrow d$. Now, since $b \Rightarrow c \leq a \Rightarrow d$ iff $a*(b \Rightarrow c) \leq d$ iff $a \leq (b \Rightarrow c) \Rightarrow d$, we obtain the desired result. \smallskip
	
	For $\mathsf{Q}(\mathrm{\Rightarrow \texttt{r}})$  we need to prove that if $d \approx 1$, then  $b \Rightarrow d \approx 1$. 
	By residuation, recall that $1 \leq b \Rightarrow d$ iff $1 * b \leq d$.  The latter equation is however obviously true, given that $d\approx 1$.
	
	The quasiequations $\mathsf{Q}(\mathrm{\land \texttt{l1}})$, $\mathsf{Q}(\mathrm{\land \texttt{l2}})$, $\mathsf{Q}(\mathrm{\land \texttt{r}})$, $\mathsf{\mathrm{Q}}(\mathrm{\lor \texttt{l1}})$, $\mathsf{Q}(\mathrm{\lor \texttt{r1}})$ and $\mathsf{Q}(\mathrm{\lor \texttt{r2}})$ follow
	straightforwardly from the fact that $\mathbf{A}^\prime$ is partially ordered and the order is determined by the implication. \smallskip
	
	To prove $\mathsf{\mathrm{Q}}(\mathrm{\lor \texttt{l2}})$, notice that $(b \lor c)^2 \leq b^2 \lor c^2$ by 
	Lemma~\ref{lemma1}.4.
	Suppose $b^2 \leq d$ and $c^2 \leq d$, then since $\mathbf{A}^\prime$ is a lattice, we have $b^2 \lor c^2 \leq d$ and 
	we conclude that $(b \lor c)^2 \leq d$ and thus $(b \lor c)^2 \Rightarrow d \approx 1$. \smallskip
	
	As to $\mathsf{Q}(\mathrm{\neg {\Rightarrow} \texttt{l}})$, {by $\mathsf{E} ( \texttt{A1} )$ we know that $b \Rightarrow b \approx 1$, therefore  we have $b * c \leq b$ and $b * c \leq c$}. Thus $b * c \leq b \land c$. Now, if $b \land c \leq d$, then $b * c \leq d$. \smallskip 
	
	To prove $\mathsf{Q}(\mathrm{\neg {\Rightarrow} \texttt{r}})$, suppose $d^2 \leq b \land c$. Using the $\leq$-monotonicity of~$*$, we have $d^2*d^2 \leq (b \land c)*(b \land c)$, i.e., $d^4 \leq (b \land c)^2$. Using 3-potency, we have $d^4 \approx d^2$, therefore $d^2 \leq (b \land c)^2$.   Since $(b \land c)^2 \leq b * c$, we have 
	$d^2 \leq (b* c)$. \smallskip
	
	$\mathsf{Q}(\mathrm{\neg {\land} \texttt{l}})$, $\mathsf{Q}(\mathrm{\neg {\land} \texttt{r}})$, $\mathsf{Q}(\mathrm{\neg {\lor} \texttt{l}})$  and $\mathsf{Q}(\mathrm{\neg {\lor} \texttt{l}})$ follow from the De Morgan's Laws \cite[Lemma 3.17]{GaJiKoOn07}. \smallskip
	%
	Finally,
	we have $\mathsf{Q}(\mathrm{\neg \neg \texttt{l}})$ and $\mathsf{Q}(\mathrm{\neg \neg \texttt{r}})$ 
	because $\mathbf{A}^\prime$ is involutive. 
	
	It remains to prove the quasiequation according to which 
	$(a \Rightarrow b) \approx 1$ and $(b \Rightarrow a) \approx 1$ imply $a \approx b$. We have that $a \leq b$ and $b \leq a$; since $\leq$ is anti-symmetric it follows  that $a \approx b$.
\end{proof}

From Propositions~\ref{prop:slinha} and~\ref{p:slinhaiss} above we obtain the desired result:

\begin{theorem}
	\label{th:s=slinha}
	The classes of  $\NS$-algebras and of $\NS'$-algebras are term-equivalent\footnote{We refer the reader to~\cite[p.~329]{SpVe08} for a formal definition of term-equivalence. Informally, Theorem~\ref{th:s=slinha} is saying that $\NS$-algebras and  $\NS'$-algebras may be seen as two equivalent presentations of the `same' class of algebras in different algebraic languages, analogous to the well-known presentation of Boolean algebras as Boolean rings or to that of $MV$-algebras as certain lattice-ordered groups.}.
	
\end{theorem}

In the next section we are going to use the algebraic insight gained through Theorem~\ref{th:s=slinha} to provide an alternative and more perspicuous axiomatization
of~$\NS$.

\section{A finite Hilbert-style calculus for $\NS$}
\label{s:four}

We are now going to  introduce a finite Hilbert-style calculus
and prove that it is algebraizable with respect to the class of $\NS'$-algebras (hence, with respect to $\NS$-algebras). 
This will give us a finite presentation of~$\NS$ that is equivalent to Nelson's calculus of Section~\ref{s:two}, but with the added advantage of involving only a finite number of axiom schemata.

Our 
calculus is an axiomatic strengthening of the \emph{full Lambek calculus with exchange and weakening} 
($\mathcal{FL}_{ew}$; see e.g.~\cite{OnoKo85}),
which will allow us to obtain the algebraizability of~$\NS$ 
as an easy extension of the
corresponding result about $\mathcal{FL}_{ew}$.

\subsection*{The calculus $\NS'$}

\begin{definition}
	The logic $\NS' := \langle \mathbf{Fm}, \vdash_{\NS'} \rangle$ is 
	the sentential
	logic in the language $\langle \land, \lor, *, \Rightarrow, 
	0, 1 \rangle $ of type $\langle 2, 2, 2, 2, 
	0, 0 \rangle $
	defined by the Hilbert-style calculus with the following axiom schemata and \textit{modus ponens} as its only rule schema:
\end{definition}

\begin{description}
	\item[\texttt{(S}1)] \qquad $(\varphi \Rightarrow \psi) \Rightarrow ((\psi \Rightarrow \gamma) \Rightarrow (\varphi \Rightarrow \gamma))$
	\item[\texttt{(S}2)] \qquad $(\varphi \Rightarrow (\psi \Rightarrow \gamma)) \Rightarrow (\psi \Rightarrow (\varphi \Rightarrow \gamma))$
	\item[\texttt{(S}3)] \qquad $\varphi \Rightarrow (\psi \Rightarrow \varphi)$
	\item[\texttt{(S}4)] \qquad $\varphi \Rightarrow (\psi \Rightarrow (\varphi * \psi))$
	\item[\texttt{(S}5)] \qquad $(\varphi \Rightarrow (\psi \Rightarrow \gamma)) \Rightarrow ((\varphi * \psi) \Rightarrow \gamma)$
	\item[\texttt{(S}6)] \qquad $(\varphi \land \psi) \Rightarrow \varphi$
	\item[\texttt{(S}7)] \qquad $(\varphi \land \psi) \Rightarrow \psi$
	\item[\texttt{(S}8)] \qquad $(\varphi \Rightarrow \psi) \Rightarrow ((\varphi \Rightarrow \gamma) \Rightarrow (\varphi \Rightarrow (\psi \land \gamma)))$
	\item[\texttt{(S}9)] \qquad $\varphi \Rightarrow (\varphi \lor \psi)$
	\item[\texttt{(S}10)] \qquad $\psi \Rightarrow (\varphi \lor \psi)$
	\item[\texttt{(S}11)] \qquad $(\varphi \Rightarrow \gamma) \Rightarrow ((\psi \Rightarrow \gamma) \Rightarrow ((\varphi \lor \psi) \Rightarrow \gamma))$
	\item[\texttt{(S}12)] \qquad $
	1
	$
	\item[\texttt{(S}13)] \qquad $0 
	\Rightarrow \varphi$
	\item[\texttt{(S}14)] \qquad 
	$((\varphi \Rightarrow  0 ) \Rightarrow 0) \Rightarrow \varphi $
	\item[\texttt{(S}15)] \qquad $(\varphi \Rightarrow (\varphi \Rightarrow (\varphi \Rightarrow \psi))) \Rightarrow (\varphi \Rightarrow (\varphi \Rightarrow \psi))$
	
\end{description}

Axioms from $(\textbf{\texttt{S1}})$ to $(\textbf{\texttt{S13}})$
are those that axiomatize $\mathcal{FL}_{ew}$
as presented in \linebreak
\cite[Section 5]{SpVe08a}, where 
$\mathcal{FL}_{ew}$ is proven to be
algebraizable. 
From that result we can immediately obtain the following:

\begin{theorem}
	\label{th:dois}
	The calculus 
	$\NS'$ is algebraizable %
	with the same defining equation and equivalence formulas as~$\NS$ 
	(cf.~Theorem~\ref{th:eqfor}).
	Its equivalent algebraic semantics is the class of $\NS'$-algebras.
\end{theorem}
\begin{proof}
	We know from  \cite[Lemmas 5.2 and 5.3]{SpVe08a} that $\mathcal{FL}_{ew}$ is algebraizable
	with respect to the class of CIBRLs. 
	Given that $\NS'$ is an axiomatic strengthening of $\mathcal{FL}_{ew}$, by
	\cite[Proposition 3.31]{Font16},
	it is also 
	algebraizable with the same defining equation and equivalence formulas. 
	The corresponding class of algebras
	is a subvariety of the class of CIBRLs, and it can be axiomatized by adding 
	equations 
	corresponding to the new axioms, as described in Def.\ \ref{def:salg}.1.
	It is easy to check that these
	imply precisely 
	that the equivalent algebraic semantics 
	of $\NS'$ is the class of all involutive  (\textbf{\texttt{S14}}) and three-potent (\textbf{\texttt{S15}}) CIBRLs, i.e., the class of $\NS'$-algebras.
\end{proof}

Although the logics $\NS$ and $\NS'$ were  
initially defined over different propositional languages
(namely $\langle \land, \lor, \Rightarrow, \neg, 0 
\rangle $ for $\NS$ and
$\langle \land, \lor, \Rightarrow, *, \neg, 0, 1 
\rangle $ for $\NS'$),
we can obviously expand the language of~$\NS$ to include the connectives
1 
and $*$  defined by 
$1 
: = \neg 0 
$ and 
$\phi * \psi := \neg (\phi \Rightarrow \neg \psi)$. This allows us to state the following:

\begin{corollary} 
	The calculi $\NS$ (in the above-defined expanded language) and $\NS'$ define the same consequence relation.
\end{corollary}
\begin{proof}
	The result follows straightforwardly from the fact that 
	$\NS$ and $\NS'$ are algebraizable (with the same defining equation and equivalence formulas) with respect to the same class of
	algebras. To add some detail one can invoke the algorithm of  \cite[Proposition 3.47]{Font16}, which allows one
	to obtain an axiomatization of an algebraizable logic from a presentation of the corresponding class 
	$K$ that comprises its equivalent
	algebraic semantics; notice that the algorithm uses only  the (quasi)equations that axiomatize $K$ and the 
	defining equations and equivalence formulas
	witnessing algebraizability. 
\end{proof}

We close the section with a non-trivial result about $\NS$ that would also not have been easily established if one had to work
with Nelson's original presentation. It is well known that substructural logics enjoy a generalized version of the Deduction-Detachment Theorem~\cite[Theorem 2.14]{GaJiKoOn07}.
Combining this result with the algebraic insight obtained in
Subsection~\ref{ss:threepointtwo} allows us to obtain a `global' deduction theorem for~$\NS$:


\begin{theorem}[Deduction-Detachment Theorem]
	\label{th:ddt}
	For all $\Gamma \cup \{ \phi\} \subseteq Fm$,
	$
	\Gamma \cup \{\varphi\} \vdash_{\NS} \psi \ \text{ if and only if  } \ 
	\Gamma \vdash_{\NS} \varphi^2  \Rightarrow \psi.
	$	
\end{theorem}
\begin{proof}
	From \cite[Corollary 2.15]{GaJiKoOn07}  we have that  $\Gamma \cup \{\varphi\} \vdash \psi$ iff   $\Gamma \vdash \varphi^{n} \Rightarrow \psi$ for some~$n$. Now it is easy to see that in~$\NS$, thanks to \texttt{(S15}),
	we can always choose $n=2$.
\end{proof}

Theorem~\ref{th:ddt} suggests that, upon defining
$\varphi  \to \psi := \varphi^2  \Rightarrow \psi$,
one may obtain in~$\NS$ a new implication-type connective $\to$ 
that enjoys the standard formulation of the Deduction-Detachment Theorem (for which $n=1$).
This is precisely what happens in Nelson's logic $\mathcal{N}3$, where
in fact $\to$ is usually taken as the primitive implication and $\Rightarrow$
as  defined  (see Subsection~\ref{ss:sixpointtwo}). 
Whether a similar interdefinability result holds for
$\NS$ as well is actually an interesting open question, to which we shall
return in Section~\ref{s:seven}. 
%
For now, what we can say is that the above-defined term $\to$ does indeed behave on $\NS$-algebras like an implication operation, at least in the abstract sense introduced by Blok, K{\"o}hler and Pigozzi~\cite{BKP-EDPC-II}.
The latter paper is the second of a series devoted to classes of algebras of non-classical logics~\cite[Definition II.9.3]{BuSa00}, focusing in particular on varieties that enjoy the property of having
\emph{equationally definable principal congruences}, or EDPC for short~\cite[p.~338]{BKP-EDPC-II}.
This is quite a strong property; in particular,
it implies congruence-distributivity and the congruence extension property
~\cite[Theorem 1.2]{BP-EDPC-I}.
It is well known that a logic that is algebraizable with respect to some variety of algebras enjoys a (generalized) 
Deduction-Detachment Theorem%
\footnote{See e.g.~\cite[Definition 3.76]{Font16} for a precise definition of generalized Deduction-Detachment Theorem.}
if and only if its associated variety has EDPC~\cite[Corollary 3.86]{Font16}. 
This applies, in particular, to our logic $\NS$ and to $\NS$-algebras.

In the context of varieties of non-classical logic having EDPC,
the authors of~\cite{BKP-EDPC-II} single out those  that possess term-definable
operations that can be viewed as generalizations of intuitionistic 
conjunction, implication and bi-implication. 
These operations are called, respectively, \emph{weak meet},
\emph{weak relative pseudo-complementation} and 
\emph{G{\"o}del equivalence}.
Algebras containing such operations are called
\emph{weak Brouwerian semilattices with filter-preserving operations}, or 
WBSO for short~\cite[Definition 2.1]{BKP-EDPC-II}. 


According to~\cite{Agliano97}, a variety having a constant~$1$ is called \emph{subtractive} if the congruences of any algebra in the variety permute at~$1$. Subtractive WBSO varieties are particularly interesting because the lattice of congruences of any algebra $\A$ belonging to a subtractive WBSO variety is isomorphic to the ideal lattice of $\A$, for a certain uniformly-defined notion of ideal.
As observed in~\cite[p.~358]{BKP-EDPC-II}, the algebraic counterpart of 
Nelson's logic $\mathcal{N}3$ is a WBSO variety. 
The same is true for $\NS$-algebras; in fact, we can here prove
a slightly stronger result:

\begin{theorem}
	\label{th:wbso}
	$\NS$-algebras form a WBSO variety in which a \emph{weak meet} is given by $\land$ (or, equivalently, by $*$),
	\emph{weak relative pseudo-complementation} is
	given by the term $x^2  \Rightarrow y$ and 
	\emph{G{\"o}del equivalence} is $x \Leftrightarrow y$.
	In fact, $\NS$-algebras form a \emph{subtractive WBSO variety}.
\end{theorem}

\begin{proof}
	One could directly check that, with the above choice of terms, $\NS$-algebras satisfy all properties of~\cite[Definition 2.1]{BKP-EDPC-II}. But we can provide a more compact proof as follows.
	As mentioned earlier, since the logic $\NS$ has a form of Deduction-Detachment Theorem (our Theorem~\ref{th:ddt}),
	we know that the variety of~$\NS$-algebras has EDPC~\cite[Corollary 3.86]{Font16}.
	We can then apply \cite[Theorem 3.3]{SpVe07} (note that $\NS$-algebras satisfy the premisses of the theorem
	thanks to~\cite[Lemma 3.2]{SpVe07}) to conclude that $(x \Leftrightarrow y )^2 * z $ is a 
	\emph{ternary deductive term} for $\NS$-algebras 
	in the sense of~\cite[Definition 2.1]{BP-EDPC-III}
	that is moreover \emph{regular with respect to~$1$}~\cite[Definition 4.1]{BP-EDPC-III}. 
	Then, by~\cite[Theorem 4.4]{BP-EDPC-III}, we have that $\NS$-algebras form a WBSO variety. 
	Finally, to check that $\NS$-algebras are subtractive~\cite[p.~214]{Agliano97}, it is sufficient to note that they satisfy the equation $1^2 \Rightarrow x \approx x$ \cite[p.~215]{Agliano97}. 
\end{proof}

Regarding the proof of the preceding theorem, it may be interesting to note that applying~\cite[Theorem 4.4]{BP-EDPC-III} to the ternary deductive term $(x \Leftrightarrow y )^2 * z $ would give us different witnessing terms: namely, we would obtain
$x^2 * y$ as weak meet, 
$ (x^2 \Rightarrow (x^2 * y))^2$ as
weak relative pseudo-complementation and 
$(x \Leftrightarrow y )^2$ as
G{\"o}del equivalence. This is not surprising, for such terms need not be unique. 

%
%

\section{More on $\NS$-algebras}
\label{s:five} 

\subsection{A non-distributive $\NS$-algebra}
\label{ss:fiveone}

We are now going to look at a particular $\NS$-algebra that provides a counter-example for
several formulas of~$\NS$ that are not valid, including the formulas which Nelson claims (without proof) not to be provable
in his calculus~\cite[p.~213]{Nel59}.

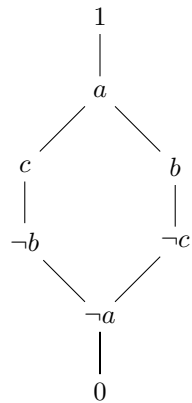
\begin{figure}[h]
	\label{a8}
	$$
	\begin{tikzpicture}
	\label{A8}[scale=.9][h]
	\node (1) at (0,2) {$1$};
	\node (6) at (0,1) {$a$};
	\node (2) at (-1,0) {$c$};
	\node (5) at (1,0) {$b$};
	\node (4) at (-1,-1) {$\neg b$};
	\node (3) at (1,-1) {$\neg c$};
	\node (7) at (0,-2) {$\neg a$};
	\node (0) at (0,-3) {$0$};

	\draw (1) -- (6) -- (5) -- (3) -- (7) -- (0) -- (7) -- (4) -- (2) -- (6);
\end{tikzpicture}
$$
\caption{$\mathbf{A_{8}}$}
\end{figure}

\medskip

\noindent 

\begin{example} The algebra $\mathbf{A_{8}}$ shown in Figure~1 
	is an  $\NS$-algebra whose lattice reduct is obviously not distributive.
	The table for the implication $\Rightarrow$ of $\mathbf{A_{8}}$ is shown below.
\end{example}
	
	
	$$\begin{tabular}{r|rrrrrrrr}
	$\Rightarrow$ & 0 & 1 & $c$ & $\neg c$ & $\neg b$ & $b$ & $a$ & $\neg a$\\
	\hline
	0 & 1 & 1 & 1 & 1 & 1 & 1 & 1 & 1 \\
	1 & 0 & 1 & $c$ & $\neg c$ & $\neg b$ & $b$ & $a$ & $\neg a$ \\
	$c$ & $\neg c$ & 1 & 1 & $b$ & $a$ & $b$ & 1 & $b$ \\
	$\neg c$ & $c$ & 1 & $c$ & 1 & $c$ & 1 & 1 & $c$ \\
	$\neg b$ & $b$ & 1 & 1 & $b$ & 1 & $b$ & 1 & $b$ \\
	$b$ & $\neg b$ & 1 & $c$ & $a$ & $c$ & 1 & 1 & $c$ \\
	$a$ & $\neg a$ & 1 & $c$ & $b$ & $c$ & $b$ & 1 & $\neg a$ \\
	$\neg a$ & $a$ & 1 & 1 & 1 & 1 & 1 & 1 & 1\\
	\end{tabular}$$ 
	
	\medskip

{In Figure~1, we consider that $\neg 1 = 0$ and that $\neg \neg x = x$, for each $x \in A_{8}$. The above sound model of~$\NS'$ was found by using the Mace4 Model Searcher \cite{Mace4}.}
We are going to check that $\mathbf{A_{8}}$
witnesses the failure of all the formulas listed below:

\begin{proposition} 
	\label{p:51}
	The following formulas cannot be proved in $\NS$.
	\begin{enumerate}
		\item $p \lor \neg p$ \hfill (Excluded Middle)
		\item $\neg( p \land \neg p)$
		\item $( p \land \neg p) \Rightarrow q$ \hfill (Ex Contradictione)
		\item $(p \Rightarrow( p \Rightarrow q)) \Rightarrow (p \Rightarrow q)$ \hfill (Contraction)
		\item $(p \Rightarrow (q \Rightarrow r)) \Rightarrow ((p \land q) \Rightarrow r)$
		\item $(p \land \neg q) \Rightarrow \neg( p \Rightarrow q)$
		\item 
		$((p \Rightarrow q) \Rightarrow q) \Rightarrow ((q \Rightarrow p) \Rightarrow p)$
		\hfill (\L ukasiewicz)
		\item $(p  \land (q \lor r )) \Rightarrow ((p  \land q) \lor(p  \land r )) $ 		\hfill (Distributivity)
		\item $((p ^{2} \Rightarrow q ) \land ((\neg q )^{2} \Rightarrow \neg p )) \Rightarrow (p  \Rightarrow q )$
		\hfill (Nelson)
	\end{enumerate}
\end{proposition}

\begin{proof}
	Thanks to \cite[Lemma 2.6]{Font16}, if $\phi$ can be proved in $\NS$, then $\V(\phi)$ = $\V(\phi \to \phi)$. As $\V(\phi \to \phi)$ = 1,  it suffices to find, for each of the above formulas, some valuation $\V: \mathbf{Fm} \to \mathbf{A_{8}}$  such that 
	$\V (\phi) \neq 1 $. 
	\\
	1. Setting $ \V( p ) = c$, we have
	$\V (p \lor \neg p) = \V(p) \lor \neg \V(p) = c \lor \neg c = a$.
	\\
	2. Setting $ \V( p ) = c$, we have $\V(\neg( p \land \neg p))$ 
	= $\neg(  c  \land \neg  c ) = \neg(  \neg a ) =   a.$ \\
	3. Let $\V( p ) = c$ and $\V( q ) = 0$. Then  $\V((p \land \neg p) \Rightarrow q) =$ 
	$ \neg a \Rightarrow 0 = a$. \\
	4. Let $\V( p ) =  c$ and $\V( q ) = 0$. Then $\V((p \Rightarrow( p \Rightarrow q)) \Rightarrow (p \Rightarrow q)) 
	=  (c \Rightarrow ( c \Rightarrow 0)) \Rightarrow (c \Rightarrow 0) = (c \Rightarrow \neg c) \Rightarrow \neg c = b \Rightarrow \neg c = a.$ \\
	5. Let $\V( p ) = \V( q ) = c$ and $\V( r ) = 0$, then $\V((p \Rightarrow (q \Rightarrow r)) \Rightarrow ((p \land q) \Rightarrow r)) =
	(c \Rightarrow (c \Rightarrow 0)) \Rightarrow ((c \land c) \Rightarrow 0)  = b \Rightarrow (c \Rightarrow 0) = b \Rightarrow \neg c = a.$ \\
	6. Let $\V( p ) = \V( q ) = c$, then 
	$ \V((p \land \neg q) \Rightarrow \neg( p \Rightarrow q))$
	=  $(c \land \neg c) \Rightarrow \neg(c \Rightarrow c) = \neg a \Rightarrow \neg 1 = a.$
	\\
	7.
	Let 
	$\V(p) = \neg c$ and $\V(q) =  c$. We have   
	%
	$
	\V(((p \Rightarrow q) \Rightarrow q) \Rightarrow ((q \Rightarrow p) \Rightarrow p)) 
	=  ((\neg c \Rightarrow  c) \Rightarrow c) \Rightarrow ((c \Rightarrow \neg c) \Rightarrow \neg c)$ 
	$= (c \Rightarrow c) \Rightarrow (b \Rightarrow \neg c) = 1 \Rightarrow a = a$.
	\\
	8.  Let $\V(p) = c$, $\V(q) = \neg c$ and $\V(r) =  \neg b$. We have $\V((p  \land (q \lor r )) \Rightarrow ((p  \land q) \lor(p  \land r )))$ = $((c \land (\neg c \lor \neg b))$ $\Rightarrow ((c \land \neg c) \lor (c \land \neg b)) = (c \land a) \Rightarrow (\neg a \lor \neg b) = c \Rightarrow \neg b = a$.
	\\
	9. Let $\V(p) = c$ and $\V(q) = \neg b$. We have $\V(((p ^{2} \Rightarrow q ) \land ((\neg q )^{2} \Rightarrow \neg p )) \Rightarrow (p  \Rightarrow q ))$ = $(\neg (c \Rightarrow \neg c)) \Rightarrow \neg b) \land ((\neg( b \Rightarrow \neg b)) \Rightarrow \neg c) \Rightarrow (c \Rightarrow \neg b) = (\neg b \Rightarrow \neg b) \land (\neg c \Rightarrow \neg c) \Rightarrow a = 1\Rightarrow a = a.$ 
\end{proof}

If we were to add the (\L ukasiewicz) formula from Proposition~\ref{p:51} 
as a new axiom schema to $\NS$ (or $\NS'$), we would obtain precisely 
the \emph{three-valued  \L ukasiewicz logic}~\cite[Chapter 4.1]{CiMuOt99}. 
No
other non-classical logic in the \L ukasiewicz family is comparable with
$\NS$  because, on the one hand, they all lack three-potency, and, on the other,
$\NS$ does not satisfy
(\L ukasiewicz), 
which is valid in all of them.

\subsection{The Galatos-Raftery doubling construction}
\label{ss:dou}

We  present here an adaptation of the construction introduced in \cite[Section 6]{GaRa04} to embed a
commutative integral residuated lattice 
into one
having an involutive negation. This  
will provide us with a simple recipe for constructing $\NS$-algebras, and will also
prove useful in studying the relation between subclasses of residuated lattices
and subclasses of   $\NS$-algebras (see Section~\ref{s:ext}).

\begin{definition} 
	\label{invol}
	Given a  	CIRL $ \mathbf{A} := \langle A, \land, \lor,  *, \Rightarrow, 1  \rangle$,
	let $\neg A := \{\neg a : a \in A\}$ be a disjoint copy of~$A$, and let $A^{*} := A \cup \neg A$. We extend the lattice order of $ \mathbf{A}$ to $A^*$ as follows. For all $a, b \in A$:
	\begin{enumerate}
		\item $a \leq_{A^{*}} b$ iff $a \leq_A b.$
		\item $\neg a \leq_{A^{*}}  b$.
		\item $\neg a \leq_{A^{*}} \neg b$ \ iff \ $b \leq_{A} a$.
	\end{enumerate}
	For each $a \in A$, we define  $\neg (\neg a) := a$. The behavior of the lattice operations is fixed according to De Morgan's laws:
	$\neg a \land \neg b := \neg (a \lor b)$ and 
	$\neg a \lor \neg b := \neg (a \land b)$. 
	The operations $*$ and $\Rightarrow$ are extended to $A^*$ as follows:
	\vspace{-0.2cm}
	$$a * \neg b 
	:= \neg (a \Rightarrow b)  \qquad \neg a * \neg b 
	:=  \neg 1 \vspace{-0.2cm}$$ 
	$$a \Rightarrow \neg b :=  \neg (a * b)  \qquad \neg a \Rightarrow \neg b := b \Rightarrow a \qquad \neg a \Rightarrow  b := 1 $$
\end{definition}


It is shown in \cite[Section 6]{GaRa04} that, if $\A$
is a CIRL, then
\NA$^{*}$ is an involutive CIBRL
into which $\A$ is embedded in the obvious way.
Moreover, we have the following:

\begin{proposition}
	\label{p:gara}
	\NA$^{*}$ is an $\NS$-algebra if and only if  \NA \ is a three-potent CIRL. 
\end{proposition}

\begin{proof}
	One direction is immediate: if $\A^*$ is an $\NS$-algebra,
	then it is three-potent, hence so is $\A$ as a $\{ \land, \lor,  *, \Rightarrow, 1  \}$-subalgebra
	of $\A^*$.
	Conversely, if $\A$ is a three-potent CIRL, since we already know   that \NA$^{*}$ is a
	CIBRL,
	it remains to show that 
	$a^2 \leq a^3$ for all $a \in \mathbf{A^*}$. For $a \in A$ the result  follows from 3-potency of $\mathbf{A}$. 
	If $a \in \neg A$, then by Definition~\ref{invol} we have
	$a^2 = \neg 1 = a^3$.
\end{proof}

The following corollary concerns \emph{implicative lattices}, namely, CIRLs where $(\land, \Rightarrow)$ forms a residuated pair, and hence $\land$  and $*$ coincide. Implicative lattices are precisely the $0$-free subreducts of Heyting algebras.

\begin{corollary} 
	\label{c:imp}
	If \NA \ is either an implicative lattice
	or an $\NS$-algebra, 
	then \NA$^{*}$ is an $\NS$-algebra.
\end{corollary}

In fact, it is not difficult to check that if $\A$ is an implicative lattice, then $\A^*$ is 
a special $\NS$-algebra known as
an \emph{$\mathcal{N}3$-lattice} (we shall deal with these structures in Section~\ref{ss:sixpointtwo}).	

\medskip
\begin{example}
	\label{nondistr}
	Consider the three-element  linearly ordered $MV$-algebra
	\cite[Definition 1.1.1]{CiMuOt99}, that we shall call  \L$_{3}$ (for \L ukasiewicz three-valued logic), defined as follows. 
	The universe is $ \{0, \frac{1}{2}, 1\}$ with the obvious lattice ordering.
	We consider  \L$_{3}$ in the algebraic language $\langle \land, \lor, *, \Rightarrow, \neg, 0, 1 \rangle$
	with the (non-lattice) operations being given by the following tables:
\end{example}
	%
	$$
	\begin{tabular}{r|rrr}
	$\Rightarrow$ & $0$ & $\frac{1}{2}$ & $1$ \\ \hline
	$0 $      & 1   & 1   & 1    \\ 
	$\frac{1}{2}$ &  $\frac{1}{2}$  & 1 & 1   \\ 
	$1$   &   $0$    &  $\frac{1}{2}$     & 1   \\ 
	\end{tabular}\qquad \begin{tabular}{r|rrr}
	$*$ & $0$ & $\frac{1}{2}$ & $1$ \\ \hline
	$0 $      & 0   & 0   & 0    \\ 
	$\frac{1}{2}$ &  $0$  & 0 & $\frac{1}{2}$   \\ 
	$1$   &   $0$    &  $\frac{1}{2}$     & 1   \\ 
	\end{tabular} \qquad \begin{tabular}{r|r}
	$\neg$ & \\ \hline
	$0 $            & 1   \\ 
	$\frac{1}{2}$ &    $\frac{1}{2}$   \\ 
	$1$    &    $0$     \\ 
	\end{tabular}$$
	%
	We note that \L$_{3}$ is an involutive CIBRL 
	(see, e.g.,~\cite[Lemma 1.1.4 and  Proposition 1.1.5]{CiMuOt99}). 
	It is also easy to check that
	\L$_{3}$ is three-potent, and so it is an $\NS$-algebra. 
	
	Applying the doubling construction to \L$_{3}$ we obtain the six-element linearly ordered
	$\NS$-algebra (\L$_{3})^*$ with universe
	$ \{\neg 1, \neg \frac{1}{2}, \neg 0, 0, \frac{1}{2}, 1\}$.
	The lattice
	operations are determined in the obvious way and the implication is given by the table below.
	
	\begin{center}
		\begin{tabular}{r|rrrrrr}
			$\Rightarrow$ & $\neg 1$ & $\neg \frac{1}{2}$ & $\neg 0$ & 0 & $\frac{1}{2}$ & 1 \\ \hline
			$\neg 1  $            & 1    & 1               & 1      & 1 & 1                & 1 \\ 
			$\neg \frac{1}{2}$ &    $\frac{1}{2}$    & 1          & 1      & 1 &     1     & 1 \\ 
			$\neg 0$    &    $0$    &    $\frac{1}{2}$     & 1      & 1 & 1      & 1 \\ 
			$0$              &   $\neg 0$    &     $ \neg 0$ &  $\neg 0$  & 1 & 1 & 1 \\ 
			$\frac{1}{2}$ & $\neg \frac{1}{2}$  & $\neg 0$ & $\neg 0$ & $\frac{1}{2}$ & 1 & 1 \\ 
			1           &   $\neg 1  $   &   $\neg \frac{1}{2} $  &  $\neg 0 $ & $0$ & $\frac{1}{2}$ & 1
		\end{tabular}
	\end{center}
	(\L$_{3})^*$
	is an example of an $\NS$-algebra which is distributive but fails to satisfy the 
	equation 
	corresponding to 
	Nelson axiom from
	Proposition~\ref{p:51}.9.
	Setting $\V(p) := \neg 0$ and $\V(q) := \neg \frac{1}{2}$, we have
	$(((\neg 0)^{2} \Rightarrow \neg \frac{1}{2}) \land ((\neg \neg \frac{1}{2})^{2} \Rightarrow \neg \neg 0)) \Rightarrow (\neg 0 \Rightarrow \neg \frac{1}{2}) 
	= 
	((\neg 1 \Rightarrow \neg \frac{1}{2}) \land ( 0 \Rightarrow 0)) \Rightarrow \frac{1}{2} 
	= 
	( 1 \land 1 ) \Rightarrow \frac{1}{2} 
	= \frac{1}{2}$.
\smallskip


\section{
	$\mathcal{N}3$ and  $\mathcal{N}4$}
\label{s:six}

As mentioned earlier, David Nelson is remembered
for having introduced, besides~$\NS$, two better-known logics:
$\mathcal{N}3$, which is usually called just 
\emph{Nelson logic}~\cite{Nel49}, and
$\mathcal{N}4$ which is known as \emph{paraconsistent Nelson logic}~\cite{AlNe84}. 
Both logics are algebraizable with respect to classes of residuated structures (called, respectively, \emph{$\mathcal{N}3$-lattices} or \emph{Nelson algebras}, and \emph{$\mathcal{N}4$-lattices}).
The question then arises
of what is precisely the relation between $\NS$ and these other logics,
or equivalently  between $\NS$-algebras, $\mathcal{N}3$-lattices and $\mathcal{N}4$-lattices.
Can we meaningfully say that one is stronger than the other?
By looking at their algebraic models, it will not be difficult to show that $\mathcal{N}3$
(which is known to be an axiomatic strengthening of $\mathcal{N}4$)
can also be viewed as an axiomatic strengthening of~$\NS$, while $\mathcal{N}4$ and $\NS$ do not seem to be comparable
in any meaningful way.
Just to fix terminology for what follows, we shall say that a logic~$\mathcal{L}^\prime$ is a \emph{conservative expansion} of a logic~$\mathcal{L}$ when the language of~$\mathcal{L}^\prime$ expands that of~$\mathcal{L}$ and yet the consequence relations of both logics coincide on the common formulas.


\subsection{
	$\mathcal{N}4$}
\label{ss:sixpointone}

\medskip

\begin{definition}
	$\mathcal{N}4 := \langle \mathbf{Fm}, \vdash_{\mathcal{N}4} \rangle$ is 
	the sentential
	logic in the language $\langle \land, \lor, \to, \neg \rangle $ of type $\langle 2, 2, 2, 1 \rangle $
	defined by the Hilbert-style calculus with the following axiom schemata and \textit{modus ponens} as its only rule schema:
	
	\begin{description}
		\item[\texttt{(N1)}] \qquad $\phi \to (\psi \to \phi)$
		\item[\texttt{(N2)}] \qquad $(\phi \to (\psi \to \gamma)) \to ((\phi \to \psi) \to (\phi \to \gamma))$
		\item[\texttt{(N3)}]\qquad  $(\phi \land \psi) \to \phi$
		\item[\texttt{(N4)}] \qquad $(\phi \land \psi) \to \psi$
		\item[\texttt{(N5)}] \qquad $(\phi \to \psi) \to ((\phi \to \gamma) \to (\phi \to (\psi \land\gamma)))$
		\item[\texttt{(N6)}] \qquad $\phi \to (\phi \lor \psi)$
		\item[\texttt{(N7)}] \qquad $\psi \to (\phi \lor \psi)$
		\item[\texttt{(N8)}] \qquad $(\phi \to \gamma) \to ((\psi \to \gamma) \to ((\phi \lor \psi) \to \gamma))$
		\item[\texttt{(N9)}] \qquad $\neg \neg \phi \leftrightarrow \phi$
		\item[\texttt{(N10)}] \quad \; $\neg (\phi \lor\psi) \leftrightarrow (\neg \phi\land\neg \psi)$
		\item[\texttt{(N11)}] \quad \; $\neg (\phi \land\psi) \leftrightarrow (\neg \phi\lor\neg \psi)$
		\item[\texttt{(N12)}] \quad \; $\neg (\phi \to \psi) \leftrightarrow (\phi\land\neg \psi)$
	\end{description}
\end{definition}
Here $\phi \leftrightarrow \psi$ abbreviates 
$(\phi \to \psi) \land ( \psi \to  \phi)$.
The implication $\to$ in $\mathcal{N}4$ is usually called \emph{weak implication}, in contrast to the \emph{strong implication} $\Rightarrow$ that is defined 
by the following term: 
\vspace{-1mm}
$$
\phi \Rightarrow \psi :=  (\phi \to \psi) \land (\neg \psi \to \neg \phi).
$$
As the notation suggests, it is the strong implication, not the weak one,
that we shall compare with the implication of~$\NS$.
This appears indeed to be the more meaningful choice, as  explained below.
%

A 
remarkable
feature of the weak implication 
of 
$\mathcal{N}4$ 
is that, on the one hand (unlike the implication of~$\NS$), 
it enjoys the Deduction-Detachment Theorem in its standard formulation:
$\Gamma \cup \{\varphi\} \vdash_{\mathcal{N}4} \psi $ if and only if  $  
\Gamma \vdash_{\mathcal{N}4} \varphi  \to \psi
$. On the other hand, 
contraposition fails
($ \varphi \to  \psi  \not  \vdash_{\mathcal{N}4} \neg \psi \to \neg \varphi$),
and
the corresponding
`weak bi-implication' 
(again unlike $\NS$, 
as axiom \texttt{(A5)} in Definition~\ref{nelsons} makes clear)
%
does not satisfy the following congruence property: 
$  \vdash_{\mathcal{N}4} \varphi \leftrightarrow \psi  $
need not imply 
$  \vdash_{\mathcal{N}4} \neg \varphi \leftrightarrow \neg \psi  $.
By contrast, the strong implication of $\mathcal{N}4$  does not enjoy the Deduction-Detachment Theorem
but (like the implication of~$\NS$) it satisfies  contraposition, and the associated bi-implication
$(\phi \Rightarrow \psi) \land (\psi \Rightarrow \phi )$
enjoys
the congruence property. The same considerations  apply to the logic $\mathcal{N}3$ considered in the next subsection. 


It is well known~\cite[Theorem 2.6]{Ri11c} that
$\mathcal{N}4$ is algebraizable (though not implicative)
with defining equation
$\mathsf{E}(\phi) := \{\phi \approx \phi \to \phi \}$ 
and equivalence formulas  
$\Delta(\phi, \psi) := \{ \phi \Rightarrow \psi, \psi \Rightarrow \phi \} $.
The implication in $\mathsf{E}(\phi)$ 
could as well be taken to be the strong one,
so 
$\mathsf{E}(\phi) := \{\phi \approx \phi \Rightarrow \phi \}$
would work too. By contrast, 
letting
$\Delta(\phi, \psi) := \{ \phi \to \psi, \psi \to \phi \} $
or the equivalent 
$\Delta(\phi, \psi) := \{ \phi \leftrightarrow \psi \} $
would not work 
precisely because of the failure of the above-mentioned
congruence property.

The equivalent algebraic semantics of $\mathcal{N}4$ is the class of $\mathcal{N}4$-lattices defined below
\cite[Definition 8.4.1]{Od08}:

\begin{definition} 
	\label{def:n4}
	An algebra $\NA := \langle A, \lor, \land, \to, \neg  \rangle$ of type of type $\langle 2, 2, 2, 1 \rangle$ is an \emph{$\mathcal{N}4$-lattice} if it satisfies the following properties:
	\vspace{-1mm}
	\begin{enumerate}
		\item $\langle A, \lor, \land, \neg  \rangle$ is a De Morgan lattice.
		\item The relation  $\preceq$ defined for all $a, b \in A$ by $a \preceq b$ iff $(a \to b) \to (a \to b) = a \to b$ is a pre-order on $\NA$.
		\item The relation $\equiv $ defined for all $a, b \in A$ as $a \equiv b$  iff $a \preceq b$ and $b \preceq a$ 
		is compatible with 
		$\land, \lor, \to$ and the quotient  $\NA_{\bowtie}$ := $\langle A, \lor, \land, \to  \rangle$/$\equiv$ is an implicative lattice. 
		\item For all $a, b \in A$, $\neg(a \to b) \equiv a  \land \neg b$.
		\item For all $a, b \in A$, $a \leq b$ iff $a \preceq b$ and $\neg b \preceq \neg a$, where $\leq$ is the lattice order of $\NA$.
	\end{enumerate}
\end{definition}
Despite this somewhat exotic definition, the class of $\mathcal{N}4$-lattices 
can actually be axiomatized by equations only~\cite[Definition 8.5.1]{Od08}. 

A simple example of an $\mathcal{N}4$-lattice is 
$\mathbf{A_4}$, shown in Figure~2
, 
whose lattice reduct is the four-element
De Morgan algebra.

\begin{figure}[h]
	\label{a4}
	
	$$\begin{tikzpicture}[scale=.9]
	\node (1) at (0,1) {$1$};
	\node (n) at (-1,0)  {$n$};
	\node (b) at (1,0) {$b$};
	\node (0) at (0,-1) {$0$};
	
	\draw (1) -- (n) -- (0) -- (b) -- (1);
\end{tikzpicture}$$
\caption{$\mathbf{A_{4}}$}
\end{figure}
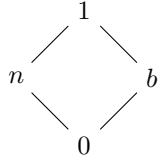
The tables for weak implication and negation in $\mathbf{A_4}$ are as follows:

$$
\begin{tabular}{l|l l l l}
$\to $  & 0  & $n$ & $b$ & 1 \\  \hline
0 & 1 & 1 & 1 & 1 \\ 
$n$ & 1 & 1 & 1 & 1 \\ 
$b$ & 0 & $n$ & $b$ & 1 \\ 
1 & 0 & $n$ & $b$ & 1  \\
\end{tabular}\qquad \qquad
\begin{tabular}{l|l}
$\neg$ &  \\ \hline
0 & 1 \\
$n$ & $n$ \\ 
$b$ & $b$ \\ 
1 & 0 \\ 
\end{tabular}$$ \\
%
One can check that $\mathbf{A_4}$ satisfies all properties of Definition~\ref{def:n4}, the quotient  $\mathbf{A_4}$/$\equiv$ mentioned in~Definition~\ref{def:n4}.3 being  the two-element Boolean algebra. 
It is also not difficult to see that no constant term is definable in $\mathbf{A_4}$. 
In fact, since the singleton $\{ b \}$ is a subuniverse
of $\mathbf{A_4}$, this element would be the only possible interpretation
for an algebraic constant. But $\{0,1 \}$ is also a subuniverse of $\mathbf{A_4}$, so 
$b$ cannot be the algebraic constant.
This implies that $\mathbf{A_4}$ has no term-definable $\NS$-algebra structure, and that
no constant term exists in the whole class of $\mathcal{N}4$-lattices.
In particular, neither  the equation $x \to x \approx y \to y $ nor $x \Rightarrow x \approx y \Rightarrow y$ hold in all
$\mathcal{N}4$-lattices.

In order to compare $\mathcal{N}4$ and $\NS$ we must fix a common propositional language, an
obvious choice being
$\langle \land, \lor, \to, \neg \rangle $, which is the primitive language of $\mathcal{N}4$ as introduced above. 
That is, we interpret the implication of~$\NS$ (up to now denoted $\Rightarrow$)
as the weak implication $\to$ of $\mathcal{N}4$. Under this interpretation, it is easy to check that for instance
the $\mathcal{N}4$ axiom \texttt{(N12)} is not 
provable
in~$\NS$
(Proposition~\ref{p:51}.6).
On the other hand, it is well known that 
the weak implication of $\mathcal{N}4$ does not satisfy 
the contraposition axiom
\texttt{(A5)} of our Definition~\ref{nelsons}:
$(\phi \to \psi) \leftrightarrow (\neg \psi \to \neg \phi)$.
Thus we must conclude that, over this language,  
$\mathcal{N}4$ and $\NS$ are incomparable.

As mentioned earlier, another possible choice for a common language would be
one that replaces~$\rightarrow$ by~$\Rightarrow$,
interpreting the original implication of~$\NS$ 
as the strong implication $\Rightarrow$ of $\mathcal{N}4$. This is also a sensible option, for
it has  been recently shown~\cite{Spin18} that the whole logic $\mathcal{N}4$
can be equivalently presented in this language: 
the weak implication
is term-definable in the $\langle \land, \lor, \Rightarrow, \neg \rangle $-fragment of   $\mathcal{N}4$ (namely, by setting $\phi \to \psi := \phi \land (((\phi \land(\psi \Rightarrow \psi)\Rightarrow \psi)\Rightarrow ((\phi \land (\psi \Rightarrow \psi)) \Rightarrow \psi))) \Rightarrow((\phi \land (\psi \Rightarrow \psi)) \Rightarrow \psi)$, see~\cite[Theorem 2.1]{Spin18}).

Under the latter interpretation, the above-mentioned contraposition axiom turns out to be valid in both logics. However, the fact that 
the equation $x \Rightarrow x \approx y \Rightarrow y$
does not hold in all 
$\mathcal{N}4$-lattices 
implies (via the algebraizability of $\mathcal{N}4$) 
that the formula 
$
(\phi \Rightarrow \phi ) \Rightarrow (\psi \Rightarrow \psi)
$, 
which is valid in $\NS$,
is not provable
in $\mathcal{N}$4. 
On the other hand,
the 
Distributivity axiom 
is valid in $\mathcal{N}4$ but not in~$\NS$, as we have seen
(Proposition~\ref{p:51}.8).
All the above arguments continue to hold also if we were to consider 
conservative expansions
of $\mathcal{N}4$ such as the logic
$\mathcal{N}4^{\bot}$ of~\cite{Od08}.

Taking into account the above observations, we conclude the following.

\begin{proposition}{$\mathcal{N}4$ (together with all of its conservative expansions) and $\NS$ are incomparable 
		over either language $\langle \land, \lor, \to, \neg \rangle $ or 
		$\langle \land, \lor, \Rightarrow, \neg \rangle $.
	}
\end{proposition}

In the next section we are going to see that at least in the case of the logic $\mathcal{N}3$ the second choice of language allows us to show that the two logics are indeed comparable, with $\NS$ being the deductively weaker among them.

\subsection{
	$\mathcal{N}3$}
\label{ss:sixpointtwo}

Nelson's logic $\mathcal{N}3 := \langle \mathbf{Fm}, \vdash_{\mathcal{N}3} \rangle$ is
the axiomatic strengthening of $\mathcal{N}4$ obtained by adding the following axiom:
\begin{description}
	\item[\texttt{(N13)}] $\neg \phi \to (\phi \to \psi)$
\end{description}

\noindent 
As an axiomatic strengthening of $\mathcal{N}4$, we have that 
$\mathcal{N}3$ is also algebraizable with the same 
defining equation and equivalence formulas.
$\mathcal{N}3$  is in fact implicative, and its equivalent algebraic semantics is the variety of 
$\mathcal{N}3$-lattices, 
which are just $\mathcal{N}4$-lattices satisfying the equation corresponding to
the above axiom (namely, $ \neg x \to (x \to y) \approx x \to x$) or, equivalently, 
$ x \to x  \approx y \to y$ (which forces integrality). 
The latter equation implies that each $\mathcal{N}3$-lattice has two algebraic constants, given by $1 := x \rightarrow x$ and $0 := \neg 1$.

In his 1959 paper \cite[p.~215]{Nel59}, Nelson mentioned that a calculus for $\mathcal{N}3$ (there denoted by $N$) could be obtained from 
his calculus for $\NS$ by removing certain rules and adding others, thus leaving it unclear whether one logic 
could be viewed as a strengthening of the other. 
Our algebraizability result for $\NS$ 
gives us a way to settle this issue.

As in the preceding subsection, we may compare $\NS$ and $\mathcal{N}3$
over the languages $\langle \land, \lor, \to, \neg, 0, 1 \rangle $ and 
$\langle \land, \lor, \Rightarrow, \neg, 0, 1 \rangle $, this time including the 
propositional constants which  are term-definable in both logics.
The first option  yields no new results, for the arguments of the preceding subsection continue to hold for 
$\mathcal{N}3$ too.
Thus, $\NS$ and $\mathcal{N}3$ are also incomparable over $\langle \land, \lor, \to, \neg, 0, 1 \rangle $.
The second option instead gives us the following.


%
%
%


\begin{proposition}
	\label{prop:propex}
	{$\mathcal{N}3$ is a (proper) strengthening of~$\NS$ over the language $\langle \land, \lor, \Rightarrow, \neg, 0, 1 \rangle $. 
	}
\end{proposition}
\begin{proof}
	It follows from~\cite[Theorem 3.12]{SpVe08} that (the $\langle \land, \lor, *, \Rightarrow, \neg, 0, 1 \rangle $-reduct of) every $\mathcal{N}3$-lattice
	satisfies all properties of our Definition~\ref{def:slinha}, and thus is an
	$\NS$-algebra.
	On the other hand, 
	$\mathcal{N}3$-lattices (like $\mathcal{N}4$-lattices) are distributive, while  $\NS$-algebras need not be.
	Thus, invoking the algebraizability of $\mathcal{N}3$ and of~$\NS$ once more, 
	we have that $\mathcal{N}3$ is a 
	\emph{proper}
	strengthening of~$\NS$.
\end{proof}

Taking into account the axiomatization of $\mathcal{N}3$ given in~\cite[p.~326]{SpVe08},
we can add further information to
the preceding proposition 
by saying that $\mathcal{N}3$ can be viewed as the \emph{axiomatic} strengthening
of~$\NS$ obtained by adding the (Distributivity) and the (Nelson) axioms from
Proposition~\ref{p:51}.
%
%
One can, in fact,
do even better, showing that an $\NS$-algebra satisfying
the equation 
corresponding to
the (Nelson) axiom must
satisfy (Distributivity) as well
\cite[Remark 3.7]{BuCi10}.
Thus we obtain the following.

\begin{proposition}
	\label{prop:n3ext}
	$\mathcal{N}3$ 
	over the language $\langle \land, \lor, \Rightarrow, \neg, 0, 1 \rangle $ is
	the axiomatic strengthening of 
	$\NS$ by the (Nelson) axiom.
\end{proposition}

It is not difficult to verify (see  Example \ref{nondistr})
that  adding 
(Distributivity) 
to $\NS$ does not allow us to prove (Nelson).
%
Thus, if we do so,
we obtain a distinct logic that is intermediate between $\NS$ and $\mathcal{N}3$. 
On the other hand,  the weakest strengthening of both 
$\NS$ and $\mathcal{N}4$ is the logic $\mathcal{N}3$ itself. 
%
To see this,
recall that $\NS$-algebras are integral residuated lattices, and therefore 
satisfy the equation $x \Rightarrow x \approx y \Rightarrow y$.
Now, an $\mathcal{N}4$-lattice satisfying such equation
(i.e.~an algebra that is at the same time an  $\mathcal{N}4$-lattice and an $\NS$-algebra)
must actually be an $\mathcal{N}3$-lattice. This can be easily checked using Odintsov's
\emph{twist-structure} representation of 
$\mathcal{N}4$-lattices~\cite[Proposition 8.4.3]{Od08}.
%
Thus,
$\mathcal{N}3$ is the join of
$\NS$ and $\mathcal{N}4$ in the lattice of all strengthenings of~$\NS$.

\begin{proposition} 
	$\mathcal{N}3\text{-lattices} = \NS\text{-algebras} \cap \mathcal{N}4\text{-lattices}$.
\end{proposition}

The following information on $\NS$-algebras is also obtained as a straightforward consequence of the previous results. 

\begin{proposition} 
	The variety of~$\NS$-algebras is not finitely generated.
\end{proposition}
\begin{proof} 
	Suppose, by contradiction, that the variety of~$\NS$-algebras
	was $V(K)$, with $K$ a finite set of finite algebras
	(see \cite[Definitions II.9.1 and II.9.4]{BuSa00}
	for the meaning of the $V, H$ and $S$  operators).
	Since $\NS$-algebras are congruence-distributive (this follows from~\cite[Theorem II.12.3]{BuSa00} and also from our
	Theorem~\ref{th:wbso}),
	we could apply J\'onsson's lemma~\cite[Corollary IV.6.10]{BuSa00} to conclude that the 
	subdirectly irreducible algebras of $V(K)$ are in $HS(K)$. Thus,
	there would be only finitely many subdirectly irreducible $\NS$-algebras.
	However, it is well known that there are infinitely many subdirectly irre\-du\-cible $\mathcal{N}3$-lattices
	(see, e.g., \cite[Corollary 9.2.11]{Od08}),
	and so there are infinitely many subdirectly irreducible $\NS$-algebras.
	%
\end{proof}



Concerning the relation between $\NS$-algebras and $\mathcal{N}3$-lattices, we can 
employ the construction introduced in Definition~\ref{invol}
to state an analogue of Proposition~\ref{p:gara}.

\begin{proposition}  
	\label{p:garan3}
	Let $\A^*$ be an $\NS$-algebra
	constructed according to Definition~\ref{invol}. 
	Then:
	\begin{enumerate}
		\item $\A^*$ is an  $\mathcal{N}3$-lattice if and only if 
		$\A$ is an implicative lattice.
		\item $\A^*$ is an $MV$-algebra (\cite[Definition 1.1.1]{CiMuOt99}) if and only if 
		$\A$ is trivial (and hence, if and only if $\A^*$ is the two-element Boolean algebra).
	\end{enumerate}
\end{proposition}  
\begin{proof} 
	1. Assume $\A$ is an implicative lattice, and let us check that the Nelson equation 
	$
	(( x ^{2} \Rightarrow  y ) \land ((\neg  y )^{2} \Rightarrow \neg  x )) \Rightarrow ( x  \Rightarrow  y ) \approx 1
	$
	holds in $\A^*$. 
	Let $a, b \in A$. Clearly, the only non-trivial cases are when either $a$ or $b$ are in $\neg A$. 
	Using the fact that $a^2 = a$ and $b^2 = b$, we have
	$(( a ^{2} \Rightarrow  \neg b ) \land ((\neg  \neg b )^{2} \Rightarrow \neg  a )) \Rightarrow ( a  \Rightarrow  \neg b )  
	=  (( a \Rightarrow  \neg b ) \land (b \Rightarrow \neg  a )) \Rightarrow ( a  \Rightarrow  \neg b ) = 1$.
	In addition, given that $\neg a \leq b$, we have
	$
	(( (\neg a) ^{2} \Rightarrow  b ) \land ((\neg  b )^{2} \Rightarrow \neg  \neg a )) \Rightarrow ( \neg a  \Rightarrow  b ) =
	(( (\neg a) ^{2} \Rightarrow  b ) \land ((\neg  b )^{2} \Rightarrow a )) \Rightarrow 1 = 1$.
	Lastly, we have
	$
	(( (\neg a) ^{2} \Rightarrow  \neg b ) \land ((\neg  \neg  b )^{2} \Rightarrow \neg  \neg a )) \Rightarrow ( \neg a  \Rightarrow  \neg  b ) =
	(( (\neg a) ^{2} \Rightarrow  \neg b ) \land (b ^{2} \Rightarrow  a )) \Rightarrow ( b  \Rightarrow  a ) 
	= (( (\neg a) ^{2} \Rightarrow  \neg b ) \land (b  \Rightarrow  a )) \Rightarrow ( b  \Rightarrow  a ) = 1$.
	
	Conversely, assume $\A^*$ is an $\mathcal{N}3$-lattice. Then,
	for all $a \in A$, we can show that $a \Rightarrow a^2 = 1$ 
	by instantiating the equation corresponding to the Nelson axiom
	as follows:
	$
	(( a ^{2} \Rightarrow  a^2 ) \land ((\neg  (a^2) )^{2} \Rightarrow \neg  a )) \Rightarrow ( a  \Rightarrow  a^2 ) = 1
	$.
	Since $ b^2 = \neg 1 $ for all $b \in \neg A$, 
	$
	(( a ^{2} \Rightarrow  a^2 ) \land ((\neg  (a^2) )^{2} \Rightarrow \neg  a )) \Rightarrow ( a  \Rightarrow  a^2 )
	= (1 \land (\neg 1 \Rightarrow \neg a)) \Rightarrow ( a  \Rightarrow  a^2 ) =  
	1 \Rightarrow ( a  \Rightarrow  a^2 ) =  a  \Rightarrow  a^2$,
	which implies the desired result.
	Thus $a = a^2$, which implies that~$\A$ is an implicative lattice.
	\\[1mm]
	2. Assume $\A^*$ is an $MV$-algebra. Then $\A^*$ satisfies the so-called divisibility equation
	$
	x * (x \Rightarrow y) = x \land y
	$ \cite[Proposition 1.1.5]{CiMuOt99}. Now let $a \in A$, so that
	$\neg a <_{A^*} a$. We have
	$\neg a = a \land \neg a  = a * (a \Rightarrow \neg a) = a * \neg (a * a) $.
	But $a * \neg a = \neg 1 $ and so, by $\leq$-monotonicity,
	$ a * \neg (a * a)  = \neg 1$. Thus, $\neg a = \neg 1$, which implies that
	$A = \{ 1 \}$.
\end{proof} 

The preceding proposition, besides characterizing the algebras of the form $\A^*$ that happen to be $\mathcal{N}3$-lattices,
highlights  the fact that $\A^*$ turns out to be an $MV$-algebra (or Heyting, or Boolean algebra) only in the degenerate case. 
For the reader familiar with the 
twist-structure representation of 
$\mathcal{N}3$-lattices~\cite[Proposition 8.4.3]{Od08}, it may be worth mentioning 
that (as one can easily verify) when
$\A^*$ is an $\mathcal{N}3$-lattice, the quotient algebra $\A^*_{\bowtie}$ given by the twist-structure representation
(cf.~Definition~\ref{def:n4})
is isomorphic to the ordinal sum of $\A$ with a one-element algebra (which constitutes the bottom element of $\A^*_{\bowtie}$).

\section{Strengthenings of~$\NS$ 
}
\label{s:ext}

In this section we take a brief look at the finitary\footnote{
	For the sake of simplicity we restrict our attention to \emph{finitary} strengthenings, though all the considerations of the present section generalize straightforwardly to arbitrary strengthenings (i.e., sub-generalized-quasivarieties).} strengthenings of~$\NS$. 
As is usual in algebraic logic, we shall in fact consider  the equivalent question about
sub
quasivarieties of~$\NS$-algebras (on
quasivarieties and quasiequations, see e.g.~\cite[Definition V.2.24]{BuSa00}). 
It is well known that 
finitary strengthenings of an algebraizable logic (in our case $\NS$) form a lattice 
that is dually isomorphic to the lattice of sub
quasivarieties of its equivalent algebraic semantics --- in our case,
$\NS$-algebras
(see e.g.~\cite[Theorem 3.33]{Font16}).
Similarly, 
axiomatic strengthenings  correspond to subvarieties~\cite[Corollary 3.40]{Font16}.

By combining, for instance, \cite[Corollary 5.3]{Ci86} with \cite[Theorem 1.59]{GaJiKoOn07}, we know that there are continuum many sub(quasi)varieties of $\mathcal{N}3$-lattices, and from this it follows that
there are at least continuum many sub
quasivarieties
of~$\NS$-algebras. An interesting  question
is how many of these sub
quasivarieties are included between $\NS$-algebras and $\mathcal{N}3$-lattices. 
Using the doubling construction of Subsection~\ref{ss:dou} we can obtain some partial results 
in this direction.

Let 
$\{ e_i : i \in I \} \cup \{ e \} \subseteq Fm \times Fm $ 
be equations in the language of residuated lattices 
(which does not include the $0$ constant).
Let
$$
q(\vec{x}) :=  \& \{ e_i : i \in I \} \ \meta \ e 
$$
be a 
quasiequation
where $\vec{x}$ are all the variables appearing in $\{ e_i : i \in I \} \cup \{ e \}$.
Define the 
quasiequation
$$
q^*(\vec{x}) := \& (\{ e_i : i \in I \} \cup \{ \neg x \sqsubseteq x : x \in \vec{x} \} ) \ \meta \ e
$$
which is built with formulas $Fm^*$ in the language of~$\NS$-algebras (which includes 0 and therefore the negation), where $x \sqsubseteq y$ is a shorthand for the equation $y \lor x \approx y$.
Notice that if
$q(\vec{x})$ is an equation
(i.e.~the set $\{ e_i : i \in I \}$ is empty), then $q^*(\vec{x})$ is a quasiequation of the form:
$$
\& \{ \neg x \sqsubseteq x : x \in \vec{x} \} ) \ \meta \ e
$$
It is not difficult to see that such a  quasiequation 
is equivalent, in the context of $\NS$-algebras, 
to the equation 
$e^*$ that is obtained from $e$ by substituting every variable~$x$ in~$e$
with the term $x \lor \neg x$. 
In other words, if  $q(\vec{x})$ is an equation in the language of residuated lattices, then
$q^*(\vec{x})$ is (equivalent to) an equation in the language of~$\NS$-algebras.

Let $\Al[A]$ be a three-potent CIRL,
and let 
$\Al[A]^*$ be the 
bounded CIRL 
obtained as in 
Definition~\ref{invol}
(which is an $\NS$-algebra by  Proposition~\ref{p:gara}).
Notice
that within $A^*$ the elements of $A$ are precisely 
the solutions to the equation $x \approx x \lor \neg x$ (abbreviated $\neg x \sqsubseteq x$),
that is,  
$
A = \{ a \in A^* : \neg a \leq a \}.
$


\begin{proposition} 	
	\label{p:trad}
	For any CIRL $\A$ and any 
	quasiequation  $q(\vec{x})$ in the language of residuated lattices,
	$$
	\A \vDash q(\vec{x}) 
	\quad
	\text{if and only if}
	\quad
	\A^* \vDash q^*(\vec{x}).
	$$ 
\end{proposition}

\begin{proof}
	For the rightward direction, it is sufficient to notice that
	all elements in $A^*$ satisfying the premisses of $q^*(\vec{x})$ must belong to $A$,
	so we can use $ q(\vec{x}) $ to obtain the desired result.
	As to the leftward direction, since any element $ a \in A$ satisfies $\neg a \leq a$, 
	we can use $q^*(\vec{x})$ to show that 
	$ q(\vec{x}) $ holds in $\A$.
\end{proof}	

Let $Q$ be a 
quasivariety 
of commutative, integral, 3-potent residuated lattices.
Then $\{ \Al[A]^* : \Al[A] \in Q \}$  is a class of~$\NS$-algebras by Proposition~\ref{p:gara}, and we can consider
the 
quasivariety $Q^* : = Q(\{ \Al[A]^* : \Al[A] \in Q \})$ generated by this class (see \cite[Definition 1.72]{Font16} for a definition of the $Q$ operator).
$Q^*$ is then a 
quasivariety  
of~$\NS$-algebras, and from our previous considerations we also know that 
if $Q$ is a 
variety, then 
$Q^*$ is also a 
variety. Moreover, from Proposition~\ref{p:trad} we have the following result:

\begin{proposition} 	
	\label{p:trad2}
	For any 
	quasivariety $Q$ of CIRLs and any 
	quasiequation $q$ in the language of $Q$, 
	we have
	$Q \vDash q$ if and only if
	$Q^* \vDash q^*$.
\end{proposition}

Denote by $\mathcal{RL}_3$ the variety of 
three-potent CIRLs,
by $\mathcal{T}$ the trivial variety (in the language of residuated lattices), and by $[\mathcal{RL}_3, \mathcal{T}]$
the lattice of all sub
quasivarieties of $\mathcal{RL}_3$.
Similarly, we denote by 
$[(\mathcal{RL}_3)^*, \mathcal{BA}]$ the interval 
(in the lattice of all sub
quasivarieties of~$\NS$-algebras)
between $(\mathcal{RL}_3)^* := Q (\{ \A^* : \A \in   \mathcal{RL}_3 
\})$
and the variety of Boolean algebras $\mathcal{BA}$. Notice 
that $\mathcal{BA} = ( \mathcal{T})^*$ by Proposition~\ref{p:garan3}.2.


\begin{proposition} 	
	\label{p:latis}
	
	The map $(\cdot)^*$ is a lattice embedding of  $[\mathcal{RL}_3, \mathcal{T}]$
	into 
	$[(\mathcal{RL}_3)^*, \mathcal{BA}]$.
\end{proposition}

\begin{proof}
	It is obvious that the map $(.)^*$ is order-preserving.
	We show that it is  also order-reflecting, which implies that it is
	injective.
	Let $Q_1, Q_2$ be  
	quasivarieties of three-potent CIRLs.
	Assume $(Q_1)^* \subseteq (Q_2)^*$,
	let $\A \in Q_1$ 
	and
	suppose $q$ is any 
	quasiequation such that $Q_2 \vDash q$.
	Then, by Proposition~\ref{p:trad}, we have $ (Q_2)^* \vDash q^*$.
	By definition we have $\A^* \in (Q_1)^*$ and therefore $\A^* \in (Q_2)^*$, which means 
	that
	$\A^* \vDash q^*$. Then again by Proposition~\ref{p:trad} we have $\A \vDash q$, which means that
	$\A \in Q_2$. Hence, $Q_1 \subseteq Q_2$ as required. 
	Thus the map $(.)^*$ is an order embedding and therefore a (complete) lattice embedding.
\end{proof}

By our previous considerations, 
$(\mathcal{RL}_3)^*$ is a variety of~$\NS$-algebras, and in fact it is  not difficult
to show that it is a proper subvariety of~$\NS$-algebras; 
it is proper because, e.g., the  equation 
$(x \Rightarrow \neg x)\lor( \neg x \Rightarrow x) \approx 1$
is valid in $(\mathcal{RL}_3)^*$
but not in all $\NS$-algebras (the algebra $\A_8$ shown earlier being a witness
).
Similarly, denoting by $\mathcal{IL}$ the variety of implicative lattices,
we have by Proposition~\ref{p:garan3}.1 that 
$(\mathcal{IL})^*$ is a proper subvariety of $\mathcal{N}3$-lattices.

By \cite[Theorem 9.54]{GaJiKoOn07}, the cardinality of   
$[\mathcal{RL}_3, \mathcal{T}]$ is greater than or equal to 
the continuum.
By Proposition~\ref{p:latis}, this implies  that there are at least continuum many 
quasivarieties in
$[(\mathcal{RL}_3)^*, \mathcal{BA}] \subseteq [\NS, \mathcal{BA}]$.
It is not difficult to see (e.g., by observing again that the equation $(x \Rightarrow \neg x)\lor( \neg x \Rightarrow x) \approx 1$
need not be satisfied in all $\mathcal{N}3$-lattices) that
$[\mathcal{N}3, \mathcal{BA}]$ is not a sublattice of 
$[(\mathcal{RL}_3)^*, \mathcal{BA}]$, and this entails that we actually 
have now some more information than when we started off.
Similarly, denoting by $\kappa$ the cardinality of 
$[\mathcal{RL}_3, \mathcal{IL}]$, we now know that
the cardinality of 
$[(\mathcal{RL}_3)^*, (\mathcal{IL})^*]$, and therefore
that of $[\NS, (\mathcal{IL})^*]$,
must be at least as large as~$\kappa$. 
Unfortunately, as far as we know, 
the cardinality of $[\mathcal{RL}_3, \mathcal{IL}]$ is at present unknown.
This leaves us with an interesting open problem,
namely the study of the cardinality (and the structure) of the lattice of logics/algebras
$[\NS, \mathcal{N}3]$. We 
mention a few more open problems in the next section.

%

\section{Future work}
\label{s:seven}

To the best of the authors' knowledge, the present paper --- together with its precursor~\cite{NaRi} --- is the first devoted to a semantical
study of Nelson's logic $\NS$. We have though but scratched the surface of what may turn out to be an interesting
topic for future research. We mention here but three directions. 

The first is to study other
types of calculi for $\NS$, for example sequent-style or display-style calculi; 
in particular one would be interested in calculi that enjoy certain desirable properties (e.g., analytic, cut-free ones) and that fit well within the general proof-theoretic framework of substructural logics. Encouraging results in this direction have been obtained about $\mathcal{N}3$, but at this point it seems far from obvious whether (or how) these may be  extended to our 
$\NS$. 

The second issue may be cast in purely algebraic terms. Thanks to recent work of M.~Spinks and R.~Veroff, 
we know that  
$\mathcal{N}4$-lattices as well as   $\mathcal{N}3$-lattices
can be equivalently presented 
taking either the strong implication $(\Rightarrow)$ or the weak one $(\to)$ as primitive. 
In the case of  $\mathcal{N}4$-lattices, this result {turns} out to be surprisingly hard to prove;
not so hard for $\mathcal{N}3$-lattices, where the term defining the weak implication
from the strong one is also simpler \cite[Theorem 1.1.3]{SpVe08}, 
\emph{viz}.\
$\phi \to \psi :=  \phi \Rightarrow ( \phi \Rightarrow  \psi)$.
The same question can now be asked about $\NS$-algebras: Is it possible to axiomatize them taking the weak implication as primitive? For this, one might start by checking which theorems of $\mathcal{N}3$ and $\mathcal{N}4$ regarding the weak implication  are valid in all $\NS$-algebras,
once we translate them according to the preceding term. 
The answer seems at the moment  far from obvious, and might provide us with further logical insight into $\NS$.
It is well known, for example, that  the $\{ \land, \lor, \to \}$-fragment of $\mathcal{N}4$ (recall that here the implication $\to$ is the weak one) coincides with the corresponding fragment of intuitionistic logic; which means that  $\mathcal{N}4$ (and  $\mathcal{N}3$) may be regarded as strengthenings of intuitionistic logic by 
an involutive De Morgan
negation. 
A solution to the above-mentioned problem would then tell us whether an analogous result can be stated 
for the logic $\NS$  as well.

Lastly, a  promising line of research may be opened by the study of the Nelson axiom/equation (Proposition~\ref{p:51}.9) in a more abstract algebraic setting. 
It is not difficult to see that 
the Nelson equation is equivalent, in the context of~$\NS$-algebras, to the following condition: 
$$
a^2 \Rightarrow b = 1   \ \text{ and } \  
(\neg b )^2 \Rightarrow \neg a = 1 
\ \ \text{ imply } \ \ 
a \leq b. 
$$
One can also (less immediately) show that this is in turn equivalent to the following:
\begin{equation}
\label{eq:ord}
\vartheta(b,1) \subseteq \vartheta(a,1)
\ \  \text{ and } \  \ 
\vartheta(a,0) \subseteq \vartheta(b,0)
\ \ \text{ imply  } \ \ 
a \leq b
\end{equation}
where $\vartheta(a,1)$ denotes the congruence generated by the set $\{ a, 1 \}$ and so on. 
It is interesting to notice that the latter condition is almost purely algebraic, 
for it  only relies on the presence of two distinguished elements in the algebra and 
(inessentially) of a 
partial order. Moreover, it closely reminds the properties 
of \emph{congruence orderability} and \emph{congruence quasi-orderability} 
studied in~\cite{Ag01, Ag01a}. This suggests that 
a purely algebraic investigation of the Nelson equation, restated as
\eqref{eq:ord}, along the lines of Aglian\`{o}'s work
may be a fruitful one. 
The first results in this direction have by now been published as~\cite{SpRiNa19,RiSp19a}. Further results
are to be found in~\cite{RiSp19b}.

\section*{Acknowledgements}

The authors would like to thank Sergey Drobyshevich 
and three anonymous referees
for several useful comments on earlier versions of the paper.


\end{document}